\documentclass[10pt,authoryear]{article}%

\usepackage[utf8]{inputenc}

\usepackage{amsmath}
\usepackage{enumitem}
\usepackage{amsfonts}
\usepackage{mathrsfs}
\usepackage{amssymb,bbm, color}
\usepackage[hidelinks,colorlinks=true,linkcolor=blue,citecolor=blue]{hyperref}
\usepackage{graphicx}
\numberwithin{equation}{section}
\usepackage[body={15.5cm,21cm}, top=3cm]{geometry}%
\providecommand{\U}[1]{\protect\rule{.1in}{.1in}}
\providecommand{\U}[1]{\protect \rule{.1in}{.1in}}
\newtheorem{theorem}{Theorem}[section]

\newtheorem{lemma}[theorem]{Lemma}

\newtheorem{proposition}[theorem]{Proposition}

\newenvironment{proof}[1][Proof]{\noindent \textbf{#1.} }{\  \rule{0.5em}{0.5em}}
\DeclareMathOperator*{\esssup}{ess\,sup}
\DeclareMathOperator*{\essinf}{ess\,inf}

\def \E{\mathsf{E}}

\def \P{\mathsf{P}}

\definecolor{purple}{rgb}{0.5,0,1}

\usepackage{natbib}
\begin{document}
	\title{Propagation of chaos for doubly mean reflected BSDEs}
	\author{Hanwu Li\thanks{Research Center for Mathematics and Interdisciplinary Sciences,
			Frontiers Science Center for Nonlinear Expectations (Ministry of Education), Shandong University, Qingdao 266237, Shandong, China.}
		\and Ning Ning\thanks{Department of Statistics,
			Texas A\&M University, College Station, Texas, USA}}
	\date{}
	\maketitle
	
	\begin{abstract}
		In this paper, we establish propagation of chaos (POC) for doubly mean reflected backward stochastic differential equations (MRBSDEs). MRBSDEs differentiate the typical RBSDEs in that the constraint is not on the paths of the solution but on its law. This unique property has garnered significant attention since the inception of MRBSDEs. Rather than directly investigating these equations, we focus on approximating them by interacting particle systems (IPS). We propose two sets of IPS having mean-field Skorokhod problems, capturing the dynamics of IPS reflected in a mean-field way.  As the dimension of the IPS tends to infinity, the POC phenomenon emerges, indicating that the system converges to a limit with independent particles, where each solves the MRBSDE. Beyond establishing the first POC result for doubly MRBSDEs, we achieve distinct convergence speeds under different scenarios.
	\end{abstract}
	
	
	\textbf{MSC-classification}: 60H10; 60K35



\allowdisplaybreaks
\newcommand{\ABS}[1]{\Bigg(#1\Bigg)} 
\newcommand{\veps}{\varepsilon} 





\section{Introduction}
We firstly give the background and motivation of doubly mean reflected backward stochastic differential equations (MRBSDEs) in Section \ref{sec:intro_background} and then state our contributions of propagation of chaos (POC) in Section \ref{sec:intro_contributions}, followed by the organization of the paper in Section \ref{sec:organization}.

\subsection{Doubly MRBSDE}
\label{sec:intro_background}
We consider a complete probability space $(\Omega, \mathcal{F}, \mathbb{F}=\{\mathcal{F}_t\}_{t\in [0,T]}, \P)$ which supports an $\mathbb{F}$-adapted standard Brownian motion $B=\{B_t\}_{t\in[0,T]}$.
In this paper, we aim to approximate the following doubly MRBSDE and establish the rate of convergence: for $t\in[0,T]$,
\begin{equation}\label{nonlinearyz_noz}
		Y_t=\xi+\int_t^T f(s,Y_s)ds-\int_t^T Z_s dB_s+K_T-K_t,
\end{equation}
with the terminal condition $\xi\in L^2(\mathcal{F}_T;\,\mathbb{R})$ and the reflection process $K_t=K^+_t-K^-_t$, satisfying that $\E[h(Y_t)]\in [l_t,u_t]$ and
\begin{equation}\label{eqn:our_Skorokhod condition}
	\int_0^T (\E[h(Y_t)]-l_t)dK_t^+=\int_0^T (u_t-\E[h(Y_t)])dK^-_t=0,
\end{equation}
where $L^2(\mathcal{F}_T;\,\mathbb{R})$ denotes the set of $\mathcal{F}_T$-measurable random variables $\xi$ taking values in $\mathbb{R}$ such that $\E|\xi|^2<\infty$, the functions $l_t, u_t$ are two obstacles and $h:\mathbb{R}\rightarrow \mathbb{R}$ is called the utility function or loss function. 
We consider a stochastic driver function
$f:\Omega\times[0,T]\times\mathbb{R}\rightarrow \mathbb{R}$ as progressively measurable for each fixed $y\in\mathbb{R}$, satisfying $\E[\int_0^T|f(\cdot,0)|^2dt]<\infty$ and the standard Lipschitz continuous condition, i.e., there exists a constant $L\geq 0$ such that 
\begin{align}
	\label{eqn:Lipschitz}
	|f(t,y)-f(t,y')|\leq L(|y-y'|),\quad \forall t\in [0,T],\quad y,y'\in \mathbb{R},\quad \P\textrm{-a.s..}
\end{align}

Equation \eqref{nonlinearyz_noz} belongs to the RBSDE family. RBSDEs naturally arose in various applications such as stochastic control, mathematical finance, and nonlinear Feynman-Kac formulae.  RSDEs were introduced by  \cite{skorokhod1961stochastic} in the case where $l=0$ and $u=+\infty$ under the typical form that $l_t\leq Y_t\leq u_t$, satisfying the following so-named Skorokhod condition:
\begin{align}
	\label{eqn:classical_Skorokhod}
	\int_0^T (Y_t-l_t)dK_t^+=\int_0^T (u_t-Y_t)dK^-_t=0.
\end{align}
RBSDEs with the above single reflection were introduced by  \cite{el1997reflected}. 
Their work was generalized by \cite{cvitanic1996backward} to RBSDEs with double reflections, where the solution lies between two general obstacles. 
Since then the interest to those doubly RBSDEs grows steadily because they are an important tool in many mathematical fields, for example \cite{hamadene2000reflected, ma2001reflected, ren2016approximate, ning2021well, ning2023multi}.

Our Skorokhod condition \eqref{eqn:our_Skorokhod condition} is different from the typical one \eqref{eqn:classical_Skorokhod} in that the constraints are on the law of the process $Y$ rather than on its paths. Hence, the RBSDE \eqref{nonlinearyz_noz} is called  MRBSDE. The mean reflection modeling was pioneered in \cite{Bouchard2015BSDEs}, where the terminal condition is replaced by a constraint on the distribution of the random variable $Y_T$. MRBSDEs were formally formulated in \cite{briand2018bsdes} with deterministic $K$ and $\E[h(Y_t)]\in [0,+\infty)$ satisfying the Skorokhod condition of mean type:
\begin{equation}\label{eqn:one_Skorokhod condition}
	\int_0^T \E[h(Y_t)]dK_t=0,
\end{equation}
which allows one to qualify the solution as flat when satisfied. Since then, MRSDEs and MRBSDEs have attracted warm interests in the probability community, which include, but are not limited to, the following: quadratic MRBSDEs \citep{hibon2017quadratic}, MRSDEs with jumps \citep{briand2020mean}, large deviation principle for the MRSDEs with jumps \citep{li2018large}, MRSDEs with two constraints \citep{falkowski2021mean}, Multi-dimensional MRBSDEs \citep{qu2023multi}, and the well-posedness of MRBSDEs with different reflection restrictions \citep{falkowski2022backward,li2023backward} that cover our doubly MRBSDE \eqref{nonlinearyz_noz}.

The focus of this paper is to approximate the doubly MRBSDE \eqref{nonlinearyz_noz} by interacting particle systems (IPS). Because the reflection process $K$ depends on the law of the position, it is hence nonlinear in the McKean–Vlasov’s terminology. Consequently, it is pertinent to explore whether this system can be regarded as the asymptotic counterpart of a mean-field Skorokhod problem – essentially, the asymptotic dynamics of a particle system reflected in a mean-field. As the dimension of the IPS approaches infinity, the system converges towards a limit where any finite subset of particles becomes asymptotically independent of each other. This phenomenon, commonly referred to as the POC, is elucidated further in the classical paper \cite{sznitman1991topics}. POC for MRSDEs was first established in \cite{Briand2020Particles} and then it was generalized to the backward framework in \cite{briand2021particles}. It's crucial to note that these two POC results specifically pertain to single reflection cases, which represent a specialized instance of the general double reflection cases. This paper is hence dedicated to the pursuit of establishing the first POC result for doubly MRBSDEs.

\subsection{Our contributions}
\label{sec:intro_contributions}
We are interested in the standard setup that $K^+,K^-\in \mathcal{I}([0,T];\,\mathbb{R})$, the set of  nondecreasing continuous functions starting from the origin. To be precise, we consider the following doubly MRBSDE:
\begin{equation}\label{nonlineary} 
	\begin{cases}
		Y_t=\xi+\int_t^T f(s,Y_s)ds-\int_t^T Z_s dB_s+K_T-K_t,\quad \E[h(Y_t)]\in [l_t,u_t],\quad  t\in[0,T], \vspace{0.2cm}\\
		K=K^+ -K^-\quad\textrm{with}\quad  K^+,K^-\in \mathcal{I}([0,T];\,\mathbb{R}),\vspace{0.2cm}\\
		\int_0^T (\E[h(Y_t)]-l_t)dK_t^+=\int_0^T (u_t-\E[h(Y_t)])dK^-_t=0.
	\end{cases}
\end{equation}
For a fixed positive integer $N$, let $\{\xi^i\}_{1\leq i\leq N}$, $\{f^i\}_{1\leq i\leq N}$, and $\{B^i\}_{1\leq i\leq N}$ be independent copies of $\xi$, $f$, and $B$, respectively. Especially, suppose that the terminal value $\xi$ and the driver $f$ of doubly MRBSDE \eqref{nonlineary} take the following form:
\begin{align}
	\label{remarkxi0}
	\xi=G(\{B_t\}_{t\in[0,T]})\quad \text{and}\quad f(t,y)=F(t,\{B_{s\wedge t}\}_{s\in[0,T]},y),
\end{align}
where $G$ and $F$ being measurable functions ensure the conditions on $\xi$ and $f$ hold, respectively. Then, for any $1\leq i\leq N$, we may take
\begin{align}
	\label{remarkxi}
	\xi^i=G(\{B^i_t\}_{t\in[0,T]})\quad \text{and}\quad f^i(t,y)=F(t,\{B^i_{s\wedge t}\}_{s\in[0,T]},y).
\end{align}
The augmented natural filtration of the family of Brownian motions $\{B^i\}_{1\leq i\leq N}$ is denoted by $\mathbb{F}^{(N)}=\{\mathcal{F}^{(N)}_t\}_{t\in[0,T]}$.

We aim to approximate the solution to the doubly MRBSDE \eqref{nonlineary}  using the following IPS:
\begin{equation}\label{eq7}
	\begin{cases}
		\overline{Y}^i_t=\theta^i+\int_t^T f^i(s,\overline{Y}^i_s) ds-\int_t^T\sum_{j=1}^N \overline{Z}^{i,j}_sd B^j_s+\overline{K}^{(N)}_T-\overline{K}^{(N)}_t, \quad 1\leq i\leq N,\vspace{0.2cm}\\
		l_t\leq \frac{1}{N}\sum_{i=1}^N h(\overline{Y}^i_t)\leq u_t, \quad t\in[0,T],\vspace{0.2cm}\\
		\overline{K}^{(N)}=\overline{K}^{(N),+}-\overline{K}^{(N),-} \quad\textrm{with} \quad \overline{K}^{(N),+},\overline{K}^{(N),-}\in \mathcal{A}^2(\mathbb{F}^{(N)},\mathbb{R}),\vspace{0.2cm}\\
		\int_0^T(\frac{1}{N}\sum_{i=1}^N h(\overline{Y}^i_t)-l_t)d\overline{K}^{(N),+}_t=\int_0^T (u_t-\frac{1}{N}\sum_{i=1}^N h(\overline{Y}^i_t))d\overline{K}^{(N),-}_t=0,
	\end{cases}
\end{equation}
where 
\begin{align}
	\label{eqn:theta_def}
	\theta^i=\xi^i+\Psi^{(N)}_T+\Phi^{(N)}_T,
\end{align}
$\Psi^{(N)}_T$, $\Phi^{(N)}_T$ are two $\mathcal{F}^{(N)}_T$-measurable random variables defined by
\begin{equation}\label{PsinPhin'}\begin{split}
		&\Psi^{(N)}_T:=\inf\Bigg\{x\geq 0;\,\frac{1}{N}\sum_{i=1}^N h(\xi^i+x)\geq l_T\Bigg\}, \\
		& \Phi^{(N)}_T:=\sup\Bigg\{x\leq 0;\,\frac{1}{N}\sum_{i=1}^N h(\xi^i+x)\leq u_T\Bigg\},
\end{split}\end{equation}
and $\mathcal{A}^2(\mathbb{F}^{(N)},\mathbb{R})$ stands for the set of all continuous, progressively measurable and nondecreasing processes $K$   such that $K_0=0$ and $\E|K_T|^2<\infty$.

Clearly, to prove the POC using the IPS \eqref{eq7}, we need to firstly establish the well-posedness of it. To this end, we first establish that of its special case, the IPS \eqref{eq7'} that the driver $f$ does not depend on $y$. That mission is accomplished in Theorem \ref{thm5.1} under the following conditions:
\vspace{0.15cm}

\noindent (A1) the functions $l,u\in \mathcal{BV}([0,T];\,\mathbb{R})$ and  
	$\inf_{t\in[0,T]}(u_t-l_t)>0$;
\vspace{0.2cm}

\noindent(A2) the function $h(x)=ax+b$ for some $a>0$ and $b\in \mathbb{R}$.
\vspace{0.2cm}

\noindent Here, $\mathcal{BV}([0,T];\,\mathbb{R})$ denotes the set of functions in $\mathcal{C}([0,T];\,\mathbb{R})$ starting from the origin having bounded variation on $[0,T]$, where $\mathcal{C}([0,T];\,\mathbb{R})$ denotes the set of continuous functions from $[0,T]$ to $\mathbb{R}$.
 Based on that, we are able to show in Theorem \ref{thm4.1} that the  IPS \eqref{eq7} has a unique solution $(\{\overline{Y}^i,\overline{Z}^i\}_{1\leq i\leq N},\overline{K}^{(N)})\in (\mathcal{S}^2(\mathbb{F}^{(N)};\,\mathbb{R})\times\mathcal{H}^2(\mathbb{F}^{(N)};\,\mathbb{R}^N))^N\times\mathcal{A}^2_{BV}(\mathbb{F}^{(N)};\,\mathbb{R})$. Here, $\mathcal{S}^2(\mathbb{F}^{(N)};\,\mathbb{R})$ stands for the set of  $\mathbb{F}^{(N)}$-adapted continuous processes $Y$ taking values in $\mathbb{R}$ such that $\E[\sup_{t\in[0,T]}|Y_t|^2]<\infty$;
$\mathcal{H}^2(\mathbb{F}^{(N)};\,\mathbb{R}^N)$ stands for the set of progressively measurable processes $Z$ taking values in $\mathbb{R}^N$ such that $\E[\int_0^T|Z_t|^2dt]<\infty$; and $\mathcal{A}^2_{BV}(\mathbb{F}^{(N)};\,\mathbb{R})$ stands for the set of all continuous, progressively measurable processes $K$  such that $K_0=0$ and $\E(|K|_0^T)^2<\infty$, where $|K|_0^T$ represents the total variation of $K$ on $[0,T]$. To achieve the POC, two crucial results are established in Proposition \ref{prop4.2'} and Lemma \ref{convergepsi}. The POC is obtained in Theorem \ref{thm4.3}, together with the convergence rate of  the solution of the IPS \eqref{eq7} to independent copies of the solution of the doubly MRBSDE \eqref{nonlineary}.

A natural question arises regarding the possibility of relaxing the linearity condition (A2). Our second POC result is tailored to address this consideration. The idea is motivated by equation (1.1) of \cite{falkowski2022backward} that the doubly MRBSDE \eqref{nonlineary} could be written as 
\begin{equation}\label{nonlinearyz} 
	\begin{cases}
		Y_t=\xi+\int_t^T f(s,  Y_s)ds-(M_T-M_t)+K_T-K_t, \quad \E[h(Y_t)]\in [l_t,u_t],\quad  t\in[0,T], \vspace{0.2cm}\\
		K=K^+ -K^- \quad\textrm{with}\quad  K^+,K^-\in \mathcal{I}([0,T];\,\mathbb{R}),\vspace{0.2cm}\\
		\int_0^T (\E[h(  Y_t)]-l_t)dK_t^+=\int_0^T (u_t-\E[h(  Y_t)])dK^-_t=0,
	\end{cases}
\end{equation}
and it has a unique solution $(Y,M,K)\in\mathcal{S}^2(\mathbb{F};\,\mathbb{R})\times\mathcal{M}_{loc}(\mathbb{F};\,\mathbb{R})\times\mathcal{A}_{BV}(\mathbb{F};\,\mathbb{R})$, 
under the following conditions:
\vspace{0.2cm}

\noindent (B1) the functions $l,u\in\mathcal{C}([0,T];\,\mathbb{R})$ and 
	$\inf_{t\in[0,T]}(u_t-l_t)>0$;
\vspace{0.2cm}

\noindent (B2) the function $h$ is strictly increasing and bi-Lipschitz, i.e., there exists $0<\gamma_l\leq \gamma_u$ \indent \hspace{0.15cm} such that for any $x,y\in\mathbb{R}$,
$\gamma_l |x-y|\leq |h(x)-h(y)|\leq \gamma_u|x-y|.
$
\vspace{0.2cm}

\noindent 
Here,
$\mathcal{M}_{loc}(\mathbb{F};\mathbb{R})$ stands for the set of continuous $\mathbb{F}$-local martingales  taking values in $\mathbb{R}$;  $\mathcal{A}_{BV}(\mathbb{F};\mathbb{R})$ stands for the set of all continuous, progressively measurable processes $K$  such that $K_0=0$ and $|K|_0^T<\infty$, $\P$-a.s..

The corresponding IPS takes the following form:
\begin{equation}\label{eq7_2_tilde}
	\begin{cases}
		\widetilde{Y}^i_t=\E\Big[\xi^i+\int_t^T f^i(s,\widetilde{Y}^i_s) ds\,\Big|\,\mathcal{F}^{(N)}_t\Big]+(\Psi^{(N)}_T+\Phi^{(N)}_T)-(\widetilde{M}^{(N)}_T-\widetilde{M}^{(N)}_t)+\widetilde{K}^{(N)}_T-\widetilde{K}^{(N)}_t,\vspace{0.2cm}\\
		l_t\leq \frac{1}{N}\sum_{i=1}^N h(\widetilde{Y}^i_t)\leq u_t, \quad t\in[0,T],\vspace{0.2cm}\\
		\widetilde{K}^{(N)}=\widetilde{K}^{(N),+}-\widetilde{K}^{(N),-}\quad\textrm{with}\quad \widetilde{K}^{(N),+},\widetilde{K}^{(N),-}\in \mathcal{A}(\mathbb{F}^{(N)};\mathbb{R}),\vspace{0.2cm}\\
		\int_0^T(\frac{1}{N}\sum_{i=1}^N h(\widetilde{Y}^i_t)-l_t)d\widetilde{K}^{(N),+}_t=\int_0^T (u_t-\frac{1}{N}\sum_{i=1}^N h(\widetilde{Y}^i_t))d\widetilde{K}^{(N),-}_t=0.
	\end{cases}
\end{equation}
Here, $\mathcal{A}(\mathbb{F}^{(N)};\,\mathbb{R})$ stands for the set of all continuous, progressively measurable and nondecreasing processes $K$ such that $K_0=0$ and $K_T<\infty$, $\P$-a.s.. After investigating the special case that the driver $f$ does not depend on $y$ in Theorem \ref{thm3.1}, we obtain the well-posedness of a unique solution $(\{\widetilde{Y}^i\}_{1\leq i\leq N},\widetilde{M}^{(N)},\widetilde{K}^{(N)})\in \mathcal{S}^2(\mathbb{F}^{(N)},\mathbb{R}^N)\times\mathcal{M}_{loc}(\mathbb{F}^{(N)};\mathbb{R})\times \mathcal{A}_{BV}(\mathbb{F}^{(N)};\mathbb{R})$ of the IPS \eqref{eq7_2_tilde} in Theorem \ref{thm4.15}, under Conditions (B1) and (B2).  Finally in Theorem \ref{thm4.35}, for $(Y^i, M^i,K)$ being an
independent copy of the solution to the doubly MRBSDE \eqref{nonlinearyz} and $(\widetilde{Y}^i,\widetilde{M}^{(N)},\widetilde{K}^{(N)})$ solving the IPS \eqref{eq7_2_tilde}, we provide the POC in terms of the convergence of $\widetilde{Y}^i$ to $Y^i$ with distinct speeds under different scenarios.

At last, we compare these two contributions.  Conditions (B1) and (B2) used in the nonlinear reflection case are weaker than Conditions (A1) and (A2) used in the linear reflection case, and correspondingly the reflection process $K$ of the solution to the IPS \eqref{eq7_2_tilde} has less desired properties than that of the IPS \eqref{eq7} since  $\mathcal{A}_{BV}^2(\mathbb{F}^{(N)};\mathbb{R})$ is a subset of $\mathcal{A}_{BV}(\mathbb{F}^{(N)};\mathbb{R})$. Furthermore, the IPS \eqref{eq7} having a $Z$ term in a vector form, has specific application interpretations in stochastic control and financial mathematics. Besides creatively designing these two IPS, in the proofs we innovatively construct the solutions in order to establish their well-posedness, by means of quantities $\overline{U}^i_t$ (resp. $\widetilde{U}^i_t$) in equation \eqref{eqn:barbarU} (resp. \eqref{eqn:tildetildeU}) and $\overline{S}_t$ (resp. $\widetilde{S}_t$) in the RBSDE \eqref{DRBSDE_overline}  (resp. \eqref{eqn:tildetildeS}). Although the proof structures are similar, the nonlinearity reflection condition (B2) loses the desired properties for quantities that could be derived by taking advantage of the linearity reflection condition (A2). We have to resort to alternative estimates obtained such as in Lemma \ref{lempsiphi}.

\subsection{Organization and notation}
\label{sec:organization}
The paper has a simple structure: we achieve POC for doubly MRBSDE with linear reflection in Section \ref{sec:linear_reflection}, and achieve POC for doubly MRBSDE with non-linear reflection in Section \ref{sec:nonlinear_reflection}.  
Throughout the paper, the letter $C$, with or without subscripts, will denote a positive constant whose value may change for different usage. Thus, $C+C = C$ and $CC = C$ are understood in an appropriate sense. Similarly, $C_{\alpha}$ denotes the generic positive constant depending on parameter $\alpha$.

\section{The case of linear reflection}
\label{sec:linear_reflection}
In this section, we first establish the well-posedness of a simplied IPS \eqref{eq7'} in Section \ref{sec:linear_constant}, based on which we establish the well-posedness of the IPS \eqref{eq7} in Section \ref{sec:linear_nonconstant}, and then we achieve POC in Section \ref{sec:linear_Approximation}.

\subsection{The particle system with constant driver}
\label{sec:linear_constant}

We aim to approximate the solution to the doubly MRBSDE \eqref{nonlineary}  using the IPS \eqref{eq7}, recalled here as follows:
\begin{equation*}
	\begin{cases}
		\overline{Y}^i_t=\theta^i+\int_t^T f^i(s,\overline{Y}^i_s) ds-\int_t^T\sum_{j=1}^N \overline{Z}^{i,j}_sd B^j_s+\overline{K}^{(N)}_T-\overline{K}^{(N)}_t, \quad 1\leq i\leq N,\vspace{0.2cm}\\
		l_t\leq \frac{1}{N}\sum_{i=1}^N h(\overline{Y}^i_t)\leq u_t, \quad t\in[0,T],\vspace{0.2cm}\\
		\overline{K}^{(N)}=\overline{K}^{(N),+}-\overline{K}^{(N),-} \quad\textrm{with} \quad \overline{K}^{(N),+},\overline{K}^{(N),-}\in \mathcal{A}^2(\mathbb{F}^{(N)},\mathbb{R}),\vspace{0.2cm}\\
		\int_0^T(\frac{1}{N}\sum_{i=1}^N h(\overline{Y}^i_t)-l_t)d\overline{K}^{(N),+}_t=\int_0^T (u_t-\frac{1}{N}\sum_{i=1}^N h(\overline{Y}^i_t))d\overline{K}^{(N),-}_t=0.
	\end{cases}
\end{equation*}
Here, $\theta^i=\xi^i+\Psi^{(N)}_T+\Phi^{(N)}_T$ is used instead of $\xi^i$ as the terminal value, since we do not have
	\begin{align*}
		l_T\leq \frac{1}{N}\sum_{i=1}^N h(\xi^i)\leq u_T
	\end{align*}
but instead 
	\begin{align*}
		l_T\leq \frac{1}{N}\sum_{i=1}^N h(\overline{Y}^i_T)=\frac{1}{N}\sum_{i=1}^N h(\xi^i+\Psi^{(N)}_T+\Phi^{(N)}_T)\leq u_T.
	\end{align*}
In fact, by the definitions of $\Psi^{(N)}$ and $\Phi^{(N)}$ in \eqref{PsinPhin'}, since $l$ and $u$ are completely separated, i.e., $\inf_{t\in[0,T]}(u_t-l_t)>0$, we have $\Psi^{(N)}_T\Phi^{(N)}_T=0$. Then, only three cases may occur: when $\Psi^{(N)}_T>0$, we have
\begin{align*}
\frac{1}{N}\sum_{i=1}^N h(\xi^i+\Psi^{(N)}_T+\Phi^{(N)}_T)=\frac{1}{N}\sum_{i=1}^N h(\xi^i+\Psi^{(N)}_T)= l_T<u_T;
\end{align*} 
when $\Psi^{(N)}_T=\Phi^{(N)}_T=0$, we have
\begin{align*}
\frac{1}{N}\sum_{i=1}^N h(\xi^i+\Psi^{(N)}_T+\Phi^{(N)}_T)=\frac{1}{N}\sum_{i=1}^N h(\xi^i)\in[l_T, u_T];
\end{align*}
when $\Phi^{(N)}_T<0$, we have
\begin{align*}
\frac{1}{N}\sum_{i=1}^N h(\xi^i+\Psi^{(N)}_T+\Phi^{(N)}_T)=\frac{1}{N}\sum_{i=1}^N h(\xi^i+\Phi^{(N)}_T)= u_T>l_T.
\end{align*}

The IPS \eqref{eq7} is a multi-dimensional doubly RBSDE, whose solution is a family of processes $(\{\overline{Y}^i,\overline{Z}^i\}_{1\leq i\leq N},K^{(N)})$ taking values in $\mathbb{R}^N\times \mathbb{R}^{N\times N}\times \mathbb{R}$.
To establish the strong well-posedness of the IPS \eqref{eq7}, we need to first study its special case that the driver $f$ does not depend on $y$, i.e.,
\begin{equation}\label{eq7'}
	\begin{cases}
		\bar{Y}^i_t=\theta^i+\int_t^T f^i(s) ds-\int_t^T\sum_{j=1}^N \bar{Z}^{i,j}_sd B^j_s+\bar{K}^{(N)}_T-\bar{K}^{(N)}_t, \  1\leq i\leq N,\vspace{0.2cm}\\
		l_t\leq \frac{1}{N}\sum_{i=1}^N h(\bar{Y}^i_t)\leq u_t, \quad t\in[0,T],\vspace{0.2cm}\\
		\bar{K}^{(N)}=\bar{K}^{(N),+}-\bar{K}^{(N),-}\quad \textrm{with}\quad \bar{K}^{(N),+},\bar{K}^{(N),-}\in \mathcal{A}^2(\mathbb{F}^{(N)},\mathbb{R}),\vspace{0.2cm}\\
		\int_0^T(\frac{1}{N}\sum_{i=1}^N h(\bar{Y}^i_t)-l_t)d\bar{K}^{(N),+}_t=\int_0^T (u_t-\frac{1}{N}\sum_{i=1}^N h(\bar{Y}^i_t))d\bar{K}^{(N),-}_t=0.
	\end{cases}
\end{equation}
Theorem \ref{thm5.1} fulfills that purpose and Proposition \ref{prop3.3} provides properties of the solution. 

\begin{theorem}\label{thm5.1}
	Under Conditions (A1) and (A2),  the IPS \eqref{eq7'} has a unique solution $(\{\bar{Y}^i,\bar{Z}^i\}_{1\leq i\leq N},\bar{K}^{(N)})\in (\mathcal{S}^2(\mathbb{F}^{(N)};\,\mathbb{R})\times\mathcal{H}^2(\mathbb{F}^{(N)};\,\mathbb{R}^N))^N\times\mathcal{A}^2_{BV}(\mathbb{F}^{(N)};\,\mathbb{R})$.
\end{theorem}


\begin{proof}
	We first prove the existence. 
	Our strategy is to construct $\bar{Y}^i_t$ as follows and show that it is a solution to the IPS \eqref{eq7'}:  
	\begin{align}
		\label{eqn:Y_decompose}
		\bar{Y}^i_t=\bar{U}^{i}_t+\bar{S}_t, \qquad 1\leq i\leq N,\qquad t\in[0,T].
	\end{align}
	Here, $\bar{U}^{i}_t$ is defined as
	\begin{align}
		\bar{U}^{i}_t:=\E\Bigg[\xi^i+\int_t^T f^i(s) ds\,\Bigg|\, \mathcal{F}^{(N)}_t\Bigg],
	\end{align}
	and $\bar{S}$ is the first component of the solution to the following RBSDE with double reflections
	\begin{equation}\label{DRBSDE}
		\begin{cases}
			\bar{S}_t=\Psi^{(N)}_T+\Phi^{(N)}_T-\int_t^T \sum_{j=1}^N\bar{Z}^{(N),j}_s dB^j_s+(\bar{K}^{(N)}_T-\bar{K}^{(N)}_t),\vspace{0.2cm}\\
			\bar{\psi}^{(N)}_t\leq \bar{S}_t\leq \bar{\phi}^{(N)}_t, \quad t\in[0,T],\vspace{0.2cm}\\
			\bar{K}^{(N)}=\bar{K}^{(N),+}-\bar{K}^{(N),-} \quad\textrm{ with }\quad \bar{K}^{(N),+},\bar{K}^{(N),-}\in \mathcal{A}^2(\mathbb{F}^{(N)},\mathbb{R}),\vspace{0.2cm}\\
			\int_0^T(\bar{S}_t-\bar{\psi}^{(N)}_t)d\bar{K}^{(N),+}_t=\int_0^T(\bar{\phi}^{(N)}_t-\bar{S}_t)d\bar{K}^{(N),-}_t=0,
		\end{cases}
	\end{equation}
	where $\bar{\psi}^{(N)}_t$ and $\bar{\phi}^{(N)}_t$ are $\mathcal{F}^{(N)}_t$-measurable random variables satisfying respectively
	\begin{align*}
		\frac{1}{N}\sum_{i=1}^N h(\bar{U}^{i}_t+\bar{\psi}^{(N)}_t)=l_t\quad\text{and}\quad \frac{1}{N}\sum_{i=1}^N h(\bar{U}^{i}_t+\bar{\phi}^{(N)}_t)=u_t.
	\end{align*}
	Since $h$ satisfies (A2), we have
	\begin{equation}\label{phinpsin}\begin{split}
			&\bar{\psi}^{(N)}_t=l_t-\frac{a}{N}\sum_{i=1}^N\bar{U}^i_t-b=l_t-b-\frac{a}{N}\sum_{i=1}^N\E\Bigg[\xi^i+\int_t^T f^i(s)ds\,\Bigg|\,\mathcal{F}^{(N)}_t\Bigg],\\
			&\bar{\phi}^{(N)}_t=u_t-\frac{a}{N}\sum_{i=1}^N\bar{U}^i_t-b=u_t-b-\frac{a}{N}\sum_{i=1}^N\E\Bigg[\xi^i+\int_t^T f^i(s)ds\,\Bigg|\,\mathcal{F}^{(N)}_t\Bigg].
	\end{split}\end{equation}
	
	By Corollary 5.5 in \cite{cvitanic1996backward}, the doubly RBSDE \eqref{DRBSDE} admits a unique solution $(\bar{S},\bar{Z}^{(N)},\bar{K}^{(N)})\in \mathcal{S}^2(\mathbb{F}^{(N)};\,\mathbb{R})\times \mathcal{H}^2(\mathbb{F}^{(N)};\,\mathbb{R}^N)\times\mathcal{A}^2_{BV}(\mathbb{F}^{(N)};\,\mathbb{R})$ and 
	\begin{align}\label{representationofS}
		\bar{S}_t=\essinf_{\sigma\in\mathcal{T}^{(N)}_{t,T}}\esssup_{\tau\in\mathcal{T}^{(N)}_{t,T}}\E\Big[\bar{R}^{(N)}(\sigma,\tau)\,\Big|\,\mathcal{F}^{(N)}_t\Big]=\esssup_{\tau\in\mathcal{T}^{(N)}_{t,T}}\essinf_{\sigma\in\mathcal{T}^{(N)}_{t,T}}\E\Big[\bar{R}^{(N)}(\sigma,\tau)\,\Big|\,\mathcal{F}^{(N)}_t\Big],  
	\end{align}
	where $\mathcal{T}^{(N)}_{t,T}$ denotes the set of all $\mathbb{F}^{(N)}$-stopping times taking values in $[t,T]$ and
	\begin{equation}
		\bar{R}^{(N)}(\sigma,\tau):=(\Psi^{(N)}_T+\Phi^{(N)}_T)\mathbbm{1}_{\{\sigma\wedge \tau=T\}}+\bar{\psi}^{(N)}_\tau \mathbbm{1}_{\{\tau<T,\tau\leq \sigma\}}+\bar{\phi}^{(N)}_\sigma \mathbbm{1}_{\{\sigma<\tau\}}
	\end{equation}
with $\Psi^{(N)}_T$ and $\Phi^{(N)}_T$ being two $\mathcal{F}^{(N)}_T$-measurable random variables defined in equation \eqref{PsinPhin'}.
Denote $\Pi^2_c(\mathbb{F}^{(N)};\,\mathbb{R})$ as the set of all continuous, nonnegative $\mathbb{F}^{(N)}$-supermartingale $M$ taking values in $\mathbb{R}$  such that $\E[\sup_{t\in[0,T]}|M_t|^2]<\infty$. 
	By equation \eqref{phinpsin}, there exist two processes $H^l,H^u\in \Pi^2_c(\mathbb{F}^{(N)};\,\mathbb{R})$, such that for $t\in[0,T]$,
	\begin{align*}
		\bar{l}^{(N)}_t\leq H^l_t-H^u_t\leq \bar{u}^{(N)}_t, 
	\end{align*}
	where \begin{equation}
		\label{eqn:lN_uN}
		\begin{split} 		
			&\bar{l}^{(N)}_t:=\bar{\psi}^{(N)}_t \mathbbm{1}_{\{t<T\}}+\Big(\Psi^{(N)}_T+\Phi^{(N)}_T\Big)\mathbbm{1}_{\{t=T\}}-\E\Big[\Psi^{(N)}_T+\Phi^{(N)}_T\,\Big|\,\mathcal{F}^{(N)}_t\Big],\\ 
			&\bar{u}^{(N)}_t:=\bar{\phi}^{(N)}_t \mathbbm{1}_{\{t<T\}}+\Big(\Psi^{(N)}_T+\Phi^{(N)}_T\Big)\mathbbm{1}_{\{t=T\}}-\E\Big[\Psi^{(N)}_T+\Phi^{(N)}_T\,\Big|\,\mathcal{F}^{(N)}_t\Big],
		\end{split}
	\end{equation}
	and for $C$ being a constant independent of $N$,
	\begin{equation}\label{hlhu}\begin{split}
		&\hspace{-0.3cm}\max(|H^l_t|,|H^u_t|)\\&\leq C\Bigg\{\frac{1}{N}\sum_{i=1}^N\E\Bigg[|\xi^i|+\int_0^T |f^i(s)|ds\,\Bigg|\,\mathcal{F}^{(N)}_t\Bigg]+\frac{1}{N}\sum_{i=1}^N\int_0^t |f^i(s)|ds+|{l}|_0^T\vee|{u}|_0^T+1\Bigg\}.
	\end{split}\end{equation}
	By the proof of Theorem 5.3 in \cite{cvitanic1996backward}, we have for any $t\in[0,T]$,
	\begin{equation}\label{barKNT}\begin{split}
		&(\bar{l}^{(N)}_t)^+\leq 
		\E\Big[\bar{K}^{(N),+}_T\,\Big|\,\mathcal{F}^{(N)}_t\Big]-\bar{K}^{(N),+}_t
		\leq H^l_t,\\
		& (\bar{u}^{(N)}_t)^-\leq \E\Big[\bar{K}^{(N),-}_T\,\Big|\,\mathcal{F}^{(N)}_t\Big]-\bar{K}^{(N),-}_t\leq H^u_t.
	\end{split}\end{equation}
	Additionally, the proof of Lemma A.1 in \cite{cvitanic1996backward} implies that for any $A \in \mathcal{A}^2(\mathbb{F}^{(N)};\,\mathbb{R})$,
	\begin{align}\label{AT}
		\E|A_T|^2\leq 4 \E\left[\sup_{t\in[0,T]}\Big|\E\Big[A_T\,\big|\,\mathcal{F}^{(N)}_t\Big]-A_t\Big|^2\right].
	\end{align}

For $\Psi^{(N)}_T$ and $\Phi^{(N)}_T$ being two $\mathcal{F}^{(N)}_T$-measurable random variables defined in equation \eqref{PsinPhin'}, it is easy to check that they are square integrable, and there exists a constant $C$ independent of $N$, such that
\begin{equation}\label{PsinPhin}
	\E|\Psi^{(N)}_T|^2\leq C\big(1+\E|\xi|^2\big)\quad \text{and}\quad \E|\Phi^{(N)}_T|^2\leq C\big(1+\E|\xi|^2\big).
\end{equation}
	By equations \eqref{phinpsin} and \eqref{eqn:lN_uN}-\eqref{PsinPhin}, there exists a constant $C$ independent of $N$, such that 
	\begin{equation}
		\label{k+k-}
		\begin{split}
			\E|\bar{K}^{(N),+}_T|^2&\leq  C\Bigg(1+\E\Bigg[|\xi|^2+\int_0^T |f(s)|^2ds\Bigg]\Bigg),\\
			\E|\bar{K}^{(N),-}_T|^2&\leq  C\Bigg(1+\E\Bigg[|\xi|^2+\int_0^T |f(s)|^2ds\Bigg]\Bigg).
		\end{split}
	\end{equation}

	By equation \eqref{eqn:Y_decompose}, we have
	\begin{align*}
		\bar{Y}^i_t&=\E\Bigg[\xi^i+\int_t^T f^i(s) ds\,\Bigg|\,\mathcal{F}^{(N)}_t\Bigg]+\E\Bigg[(\Psi^{(N)}_T+\Phi^{(N)}_T)+\bar{K}^{(N),+}_T-\bar{K}^{(N),-}_T\,\Bigg|\,\mathcal{F}^{(N)}_t\Bigg]\\
		&\hspace{9cm}-(\bar{K}^{(N),+}_t-\bar{K}^{(N),-}_t)\\
		&=\E\Bigg[\theta^i+\int_0^T f^i(s) ds +\bar{K}^{(N),+}_T-\bar{K}^{(N),-}_T\,\Bigg|\,\mathcal{F}^{(N)}_t\Bigg]-\int_0^t f^i(s)ds-(\bar{K}^{(N),+}_t-\bar{K}^{(N),-}_t),
	\end{align*}
 where we reorganized terms and used equation \eqref{eqn:theta_def}.
	By  equations \eqref{PsinPhin} and \eqref{k+k-}, we have 
	\begin{align*}
		\theta^i+\int_0^T f^i(s) ds +\bar{K}^{(N),+}_T-\bar{K}^{(N),-}_T\in L^2(\mathcal{F}^{(N)}_T;\mathbb{R}).
	\end{align*}
	By the martingale representation theorem, there exists $\bar{Z}^i\in \mathcal{H}^2(\mathbb{F}^{(N)};\,\mathbb{R}^N)$ such that 
	\begin{align*}
		\bar{Y}^i_t=&\E\Bigg[\theta^i+\int_0^T f^i(s) ds +\bar{K}^{(N),+}_T-\bar{K}^{(N),-}_T\Bigg]+\int_0^t \sum_{j=1}^N\bar{Z}^{i,j}_sdB^j_s-\int_0^tf^i(s)ds\\
		&\hspace{7cm}-(\bar{K}^{(N),+}_t-\bar{K}^{(N),-}_t).
	\end{align*}
	
	It remains to show that $\{\bar{Y}^i\}_{1\leq i\leq N}$ satisfies the IPS \eqref{eq7'} and $\bar{K}^{(N),+}$, $\bar{K}^{(N),-}$ satisfy the Skorokhod condition. Since $h$ is nondecreasing and $\bar{\psi}^{(N)}_t\leq \bar{S}_t\leq \bar{\phi}^{(N)}_t$ for $t\in[0,T]$, it is easy to check that
	\begin{align*}
		l_t=\frac{1}{N}\sum_{i=1}^N h(\bar{U}^{i}_t+\bar{\psi}^{(N)}_t)&\leq \frac{1}{N}\sum_{i=1}^N h(\bar{U}^i_t+\bar{S}_t)\\
		&=\frac{1}{N}\sum_{i=1}^N h(\bar{Y}^i_t)\leq \frac{1}{N}\sum_{i=1}^N h(\bar{U}^{i}_t+\bar{\phi}^{(N)}_t)=u_t.
	\end{align*}
	Applying the Skorokhod condition in the doubly RBSDE \eqref{DRBSDE} yields that
	\begin{align*}
		\int_0^T\Bigg(\frac{1}{N}\sum_{i=1}^N h(\bar{Y}^i_t)-l_t\Bigg)d\bar{K}^{(N),+}_t
		=&\int_0^T \Bigg(\frac{1}{N}\sum_{i=1}^N h(\bar{U}^{i}_t+\bar{S}_t)-l_t\Bigg)d\bar{K}^{(N),+}_t\\
		=&\int_0^T \Bigg(\frac{1}{N}\sum_{i=1}^N h(\bar{U}^{i}_t+\bar{S}_t)-l_t\Bigg)\mathbbm{1}_{\{\bar{S}_t=\bar{\psi}^{(N)}_t\}}d\bar{K}^{(N),+}_t\\
		=&\int_0^T\Bigg(\frac{1}{N}\sum_{i=1}^N h(\bar{U}^{i}_t+\bar{\psi}^{(N)}_t)-l_t\Bigg)\mathbbm{1}_{\{\bar{S}_t=\bar{\psi}^{(N)}_t\}}d\bar{K}^{(N),+}_t\\
		=&0,
	\end{align*}
	and similarly we have
	\begin{align*}
		\int_0^T\frac{1}{N}\Bigg(u_t-\sum_{i=1}^N h(\bar{Y}^i_t)\Bigg)d\bar{K}^{(N),-}_t=0
	\end{align*}
 as desired.
 
Now, we are ready to prove the uniqueness.
	Suppose that $(\bar{\bar{Y}}^i,\bar{\bar{Z}}^i\}_{1\leq i\leq N},\bar{\bar{K}}^{(N)})$ is another solution to the IPS \eqref{eq7'}. Without loss of generality, we assume that there exists a pair $(i,t)\in\{1,\cdots,N\}\times[0,T)$, such that $\P(\bar{A})>0$ where $\bar{A}=\{\bar{\bar{Y}}^i_t>\bar{Y}^i_t\}$. 
 Define the stopping time 
	\begin{align*}
		\bar{\tau}:=\inf\Big\{u\geq t:\bar{\bar{Y}}^i_u=\bar{Y}^i_u\Big\}.
	\end{align*}
	Then, on the set $\bar{A}$, we have $t<\bar{\tau}\leq T$ and $\bar{\bar{Y}}^i_s>\bar{Y}^i_s$ for $s\in[t,\bar{\tau})$. Since $h$ is strictly increasing, on the set $\bar{A}$, for any $s\in[t,\bar{\tau})$ we have
	\begin{align*}
		u_s\geq \frac{1}{N}\sum_{j=1}^Nh(\bar{\bar{Y}}^j_s)>\frac{1}{N}\sum_{j=1}^Nh(\bar{Y}^j_s)\geq l_s. 
	\end{align*}
	It follows from the Skorokhod condition that $d\bar{\bar{K}}^{(N),+}_s=0$ and $d\bar{K}^{(N),-}_s=0$ for $s\in[t,\bar{\tau})$ on the set $\bar{A}$. It is clear that 
	\begin{align*}
		\bar{Y}^i_t-\bar{\bar{Y}}^i_t=&\bar{Y}^i_\tau-\bar{\bar{Y}}^i_\tau-\int_t^\tau \sum_{j=1}^N(\bar{Z}^{i,j}_s-\bar{\bar{Z}}^{i,j}_s)dB^j_s+(\bar{K}^{(N),+}_\tau-\bar{K}^{(N),+}_t)-(\bar{K}^{(N),-}_\tau-\bar{K}^{(N),-}_t)\\
		&-(\bar{\bar{K}}^{(N),+}_\tau-\bar{\bar{K}}^{(N),+}_t)+(\bar{\bar{K}}^{(N),-}_\tau-\bar{\bar{K}}^{(N),-}_t)\\
		=&-\int_t^\tau \sum_{j=1}^N(\bar{Z}^{i,j}_s-\bar{\bar{Z}}^{i,j}_s)dB^j_s+(\bar{K}^{(N),+}_\tau-\bar{K}^{(N),+}_t)-(\bar{K}^{(N),-}_\tau-\bar{K}^{(N),-}_t)\\
		&-(\bar{\bar{K}}^{(N),+}_\tau-\bar{\bar{K}}^{(N),+}_t)+(\bar{\bar{K}}^{(N),-}_\tau-\bar{\bar{K}}^{(N),-}_t).
	\end{align*}
	 Then
	\begin{align*}
		(\bar{Y}^i_t-\bar{\bar{Y}}^i_t)\mathbbm{1}_{\bar{A}}=&-\int_t^\tau \sum_{j=1}^N(\bar{Z}^{i,j}_s-\bar{\bar{Z}}^{i,j}_s)dB^j_s\mathbbm{1}_{\bar{A}}+(\bar{K}^{(N),+}_\tau-\bar{K}^{(N),+}_t)\mathbbm{1}_{\bar{A}}\\
  &+(\bar{\bar{K}}^{(N),-}_\tau-\bar{\bar{K}}^{(N),-}_t)\mathbbm{1}_{\bar{A}}.
	\end{align*}
	Taking expectations on both sides, noting that $\bar{A}\in\mathcal{F}^{(N)}_t$ and $\bar{K}^{(N),+}$, $\bar{\bar{K}}^{(N),-}$ are nondecreasing, we have
	\begin{align*}
		0>\E\Big[(\bar{Y}^i_t-\bar{\bar{Y}}^i_t)\mathbbm{1}_{\bar{A}}\Big]=\E\Big[(\bar{K}^{(N),+}_\tau-\bar{K}^{(N),+}_t)\mathbbm{1}_{\bar{A}}+(\bar{\bar{K}}^{(N),-}_\tau-\bar{\bar{K}}^{(N),-}_t)\mathbbm{1}_{\bar{A}}\Big]\geq 0,
	\end{align*}
	which is a contradiction. Hence, $\bar{Y}^i\equiv\bar{\bar{Y}}^i$ for $1\leq i\leq N$. Since $\bar{Y}^i$ is an $\mathbb{F}^{(N)}$-semimartingale, its decomposition is unique. Therefore, we have $\bar{K}^{(N)}=\bar{\bar{K}}^{(N)}$ and $\bar{Z}=\bar{\bar{Z}}$, which completes the proof.
\end{proof}


\begin{proposition}\label{prop3.3}
	Under Conditions (A1) and (A2), there exists a constant $C$ independent of $N$ such that for all $1\leq i\leq N$,
	\begin{align*}
		\E\left[\sup_{t\in[0,T]}|\bar{Y}^i_t|^2\right]+\E\left[\int_0^T |\bar{Z}^i_t|^2dt\right]&+\E|\bar{K}^{(N),+}_T|^2+\E|\bar{K}^{(N),-}_T|^2\\
		\leq &C\Bigg(1+\E|\xi|^2+\E\left[\int_0^T|f(t)|^2dt\right]\Bigg),
	\end{align*}
 where $(\{\bar{Y}^i,\bar{Z}^i\}_{1\leq i\leq N},\bar{K}^{(N)})$ is the solution to the IPS \eqref{eq7'}.
\end{proposition}

\begin{proof}
	By the proof of Theorem \ref{thm5.1}, we have 
	\begin{align*}
		|\bar{Y}^i_t|&\leq |\bar{U}^{i}_t|+|\bar{S}_t|\\
		&\leq \E\Bigg[\Bigg|\xi^i+\int_t^T f^i(s) ds\Bigg|\,\Bigg|\,\mathcal{F}^{(N)}_t\Bigg]+\E\Bigg[\sup_{t\in[0,T]}|\bar{\psi}^{(N)}_t|+\sup_{t\in[0,T]}|\bar{\phi}^{(N)}_t|\,\Bigg|\,\mathcal{F}^{(N)}_t\Bigg].
	\end{align*}
	By equation \eqref{phinpsin}, the estimate for $\bar{Y}^i$ follows from Doob's inequality and H\"{o}lder's inequality. The estimate for $\bar{K}^{(N),+}$ and $\bar{K}^{(N),-}$ are given in equation \eqref{k+k-}. Applying It\^{o}'s formula, it is easy to check that 
	\begin{align*}
		|\bar{Y}^i_0|^2+\int_0^T |\bar{Z}^i_s|^2 ds=|\xi^i|^2+2\int_0^T \bar{Y}^i_s f^i(s)ds+2\int_0^T \bar{Y}^i_sd\bar{K}^{(N)}_s-2\int_0^T \bar{Y}^i_s\sum_{j=1}^{N}\bar{Z}^{i,j}_sdB^j_s.
	\end{align*}
	Taking expectations on both sides yields that 
	\begin{align*}
		&\hspace{-0.3cm}\E \left[\int_0^T |\bar{Z}^i_s|^2 ds\right] \\
		\leq &\E\Bigg[|\xi^i|^2+2\int_0^T \bar{Y}^i_s f^i(s)ds+2\int_0^T \bar{Y}^i_sd\bar{K}^{(N)}_s\Bigg]\\
		\leq &\E|\xi|^2+2\E\Bigg[\sup_{t\in[0,T]}|\bar{Y}^i_t|\int_0^T|f^i(s)|ds\Bigg]+2\E\Bigg[\sup_{t\in[0,T]}|\bar{Y}^i_s|\Big(|\bar{K}^{(N),+}_T|+|\bar{K}^{(N),-}_T|\Big)\Bigg]\\
		\leq &\E|\xi|^2+2\E \left[\sup_{t\in[0,T]}|\bar{Y}^i_t|^2\right] +T\E\left[ \int_0^T|f^i(s)|^2ds\right] +\E \Big(|\bar{K}^{(N),+}_T|+|\bar{K}^{(N),-}_T|\Big)^2 ,
	\end{align*}
 where we used H\"{o}lder's inequality in the last inequality.
	Using the estimates for $\bar{Y}^i$, $\bar{K}^{(N),+}$ and $\bar{K}^{(N),-}$, we obtain the desired result.
\end{proof}

\subsection{The particle system with non-constant driver}
\label{sec:linear_nonconstant}

In this subsection, we consider the general IPS \eqref{eq7}.
\begin{theorem}\label{thm4.1}
	Under Conditions (A1) and (A2),  the  IPS \eqref{eq7} has a unique solution $(\{\overline{Y}^i,\overline{Z}^i\}_{1\leq i\leq N},\overline{K}^{(N)})\in (\mathcal{S}^2(\mathbb{F}^{(N)};\,\mathbb{R})\times\mathcal{H}^2(\mathbb{F}^{(N)};\,\mathbb{R}^N))^N\times\mathcal{A}^2_{BV}(\mathbb{F}^{(N)};\,\mathbb{R})$. 
\end{theorem}

\begin{proof}
By Theorem \ref{thm5.1}, for a given $\overleftarrow{V}^i\in \mathcal{S}^2(\mathbb{F}^{(N)},\mathbb{R})$, the following IPS:
	\begin{displaymath}
		\begin{cases}
		\overleftarrow{Y}^i_t=\theta^i+\int_t^T f^i(s,\overleftarrow{V}^i_s) ds-\int_t^T \sum_{j=1}^N \overleftarrow{Z}^{i,j}_sd B^j_s+\overleftarrow{K}^{(N)}_T-\overleftarrow{K}^{(N)}_t, \quad 1\leq i\leq N,\vspace{0.2cm}\\
			l_t\leq \frac{1}{N}\sum_{i=1}^N h(\overleftarrow{Y}^i_t)\leq u_t, \quad t\in[0,T],\vspace{0.2cm}\\
			\overleftarrow{K}^{(N)}=\overleftarrow{K}^{(N),+}-\overleftarrow{K}^{(N),-} \quad \textrm{with}\quad \overleftarrow{K}^{(N),+},\overleftarrow{K}^{(N),-}\in \mathcal{A}^2(\mathbb{F}^{(N)},\mathbb{R}),\vspace{0.2cm}\\
			\int_0^T(\frac{1}{N}\sum_{i=1}^N h(\overleftarrow{Y}^i_t)-l_t)d\overleftarrow{K}^{(N),+}_t=\int_0^T (u_t-\frac{1}{N}\sum_{i=1}^N h(\overleftarrow{Y}^i_t))d\overleftarrow{K}^{(N),-}_t=0,
		\end{cases}
	\end{displaymath}
 admits a unique solution $(\{\overleftarrow{Y}^i,\overleftarrow{Z}^i\}_{1\leq i\leq N},\overleftarrow{K}^{(N)})\in (\mathcal{S}^2(\mathbb{F}^{(N)},\mathbb{R})\times\mathcal{H}^2(\mathbb{F}^{(N)};\,\mathbb{R}^N))^N\times\mathcal{A}^2_{BV}(\mathbb{F}^{(N)};\,\mathbb{R})$.
	We define the mapping $\Gamma:\mathcal{S}^2(\mathbb{F}^{(N)},\mathbb{R})\rightarrow\mathcal{S}^2(\mathbb{F}^{(N)},\mathbb{R})$ by
$\Gamma(\overleftarrow{V}^i)=\overleftarrow{Y}^i.
$
Similarly, by Theorem \ref{thm5.1}, 
for a given $\overrightarrow{V}^i\in \mathcal{S}^2(\mathbb{F}^{(N)},\mathbb{R})$, the following IPS:
	\begin{displaymath}
		\begin{cases}
		\overrightarrow{Y}^i_t=\theta^i+\int_t^T f^i(s,\overrightarrow{V}^i_s) ds-\int_t^T \sum_{j=1}^N \overrightarrow{Z}^{i,j}_sd B^j_s+\overrightarrow{K}^{(N)}_T-\overrightarrow{K}^{(N)}_t, \quad 1\leq i\leq N,\vspace{0.2cm}\\
			l_t\leq \frac{1}{N}\sum_{i=1}^N h(\overrightarrow{Y}^i_t)\leq u_t, \quad t\in[0,T],\vspace{0.2cm}\\
			\overrightarrow{K}^{(N)}=\overrightarrow{K}^{(N),+}-\overrightarrow{K}^{(N),-}\quad \textrm{with}\quad\overrightarrow{K}^{(N),+},\overrightarrow{K}^{(N),-}\in \mathcal{A}^2(\mathbb{F}^{(N)},\mathbb{R}),\vspace{0.2cm}\\
			\int_0^T(\frac{1}{N}\sum_{i=1}^N h(\overrightarrow{Y}^i_t)-l_t)d\overrightarrow{K}^{(N),+}_t=\int_0^T (u_t-\frac{1}{N}\sum_{i=1}^N h(\overrightarrow{Y}^i_t))d\overrightarrow{K}^{(N),-}_t=0,
		\end{cases}
	\end{displaymath}
 admits a unique solution $(\{\overrightarrow{Y}^i,\overrightarrow{Z}^i\}_{1\leq i\leq N},\overrightarrow{K}^{(N)})\in (\mathcal{S}^2(\mathbb{F}^{(N)},\mathbb{R})\times\mathcal{H}^2(\mathbb{F}^{(N)};\,\mathbb{R}^N))^N\times\mathcal{A}^2_{BV}(\mathbb{F}^{(N)};\,\mathbb{R})$.
 Then the mapping is $\Gamma(\overrightarrow{V}^i)=\overrightarrow{Y}^i$. 

  By the strategy used in the proof of Theorem \ref{thm5.1}, we construct $\overleftarrow{Y}^i_t$ and $\overrightarrow{Y}^i_t$ for any $1\leq i\leq N$ and $t\in[0,T]$, as follows:
	\begin{align*}
		\overleftarrow{Y}^i_t=\overleftarrow{U}^i_t+\overleftarrow{S}_t \quad\text{and}\quad \overrightarrow{Y}^i_t=\overrightarrow{U}^i_t+\overrightarrow{S}_t,
	\end{align*}
where
	\begin{align*}
		\overleftarrow{U}^i_t:=\E\bigg[\xi^i+\int_t^T f^i(s,\overleftarrow{V}_s^i)ds\,\bigg|\,\mathcal{F}^{(N)}_t\bigg]\quad\text{and}\quad \overrightarrow{U}^i_t:=\E\bigg[\xi^i+\int_t^T f^i(s,\overrightarrow{V}_s^i)ds\,\bigg|\,\mathcal{F}^{(N)}_t\bigg],
	\end{align*}
and $\overleftarrow{S}$ (resp. $\overrightarrow{S}$) is the first component of the solution to the doubly RBSDE \eqref{DRBSDE} with terminal value $\Phi^{(N)}_T+\Psi^{(N)}_T$, no driver,  lower obstacle $\overleftarrow{\psi}^{(N)}$ (resp. $\overrightarrow{\psi}^{(N)}$) and upper obstacle $\overleftarrow{\phi}^{(N)}$ (resp. $\overrightarrow{\phi}^{(N)}$).
Here, $\overleftarrow{\psi}^{(N)}_t$, $\overleftarrow{\phi}^{(N)}_t$, $\overrightarrow{\psi}^{(N)}_t$ and $\overrightarrow{\phi}^{(N)}_t$ are $\mathcal{F}^{(N)}_t$-measurable random variables such that
	\begin{align*}
		&\frac{1}{N}\sum_{i=1}^N h(\overleftarrow{U}^i_t+\overleftarrow{\psi}^{(N)}_t)=l_t, \qquad\quad \frac{1}{N}\sum_{i=1}^N h(\overleftarrow{U}^i_t+\overleftarrow{\phi}^{(N)}_t)=u_t,\\
		&\frac{1}{N}\sum_{i=1}^N h(\overrightarrow{U}^i_t+\overrightarrow{\psi}^{(N)}_t)=l_t, \qquad \quad\frac{1}{N}\sum_{i=1}^N h(\overrightarrow{U}^i_t+\overrightarrow{\phi}^{(N)}_t)=u_t.
	\end{align*}
 
It follows from the construction of $\overleftarrow{Y}^i$, $\overrightarrow{Y}^i$ and the representation for $\overleftarrow{S}$, $\overrightarrow{S}$ (see Corollary 5.5 in \cite{cvitanic1996backward}) that
	\begin{align*}
		|\overleftarrow{Y}^i_t-\overrightarrow{Y}^i_t|\leq &|\overleftarrow{U}^i_t-\overrightarrow{U}^i_t|+|\overleftarrow{S}_t-\overrightarrow{S}_t|\\
		\leq & LT\E\Bigg[\sup_{s\in[0,T]}|\overleftarrow{V}^i_s-\overrightarrow{V}^i_s|\,\Bigg|\,\mathcal{F}^{(N)}_t\Bigg]+\E\Bigg[\sup_{s\in[0,T]}|\overleftarrow{\psi}^{(N)}_s-\overrightarrow{\psi}^{(N)}_s|\,\Bigg|\,\mathcal{F}^{(N)}_t\Bigg]\\
  &+\E\Bigg[\sup_{s\in[0,T]}|\overleftarrow{\phi}^{(N)}_s-\overrightarrow{\phi}^{(N)}_s|\,\Bigg|\,\mathcal{F}^{(N)}_t\Bigg].
	\end{align*}
	Applying Doob's inequality, H\"{o}lder's inequality and equation \eqref{phinpsin} yields
	\begin{align*}
		&\hspace{-1cm}\E\Bigg[\sup_{t\in[0,T]}|\overleftarrow{Y}^i_t-\overrightarrow{Y}^i_t|^2\Bigg]\\
  \leq& 12L^2T^2\E\Bigg[\sup_{t\in[0,T]}|\overleftarrow{V}^i_t-\overrightarrow{V}^i_t|^2\Bigg]+12\E\Bigg[\sup_{t\in[0,T]}|\overleftarrow{\psi}^{(N)}_t-\overrightarrow{\psi}^{(N)}_t|^2\Bigg]\\
  &+12\E\Bigg[\sup_{t\in[0,T]}|\overleftarrow{\phi}^{(N)}_t-\overrightarrow{\phi}^{(N)}_t|^2\Bigg]\\
		\leq &12L^2T^2\E\Bigg[\sup_{t\in[0,T]}|\overleftarrow{V}^i_t-\overrightarrow{V}^i_t|^2\Bigg]+24a^2\E\Bigg[\sup_{t\in[0,T]}\Bigg(\frac{1}{N}\sum_{j=1}^N|\overleftarrow{U}^j_t-\overrightarrow{U}^j_t|\Bigg)^2\Bigg]\\
		\leq& 12L^2T^2\E\Bigg[\sup_{t\in[0,T]}|\overleftarrow{V}^i_t-\overrightarrow{V}^i_t|^2\Bigg]+96a^2L^2T^2\E\Bigg[\frac{1}{N}\sum_{i=1}^N\sup_{t\in[0,T]}|\overleftarrow{V}^i_t-\overrightarrow{V}^i_t|^2\Bigg],
	\end{align*}
	where $a$ is from Condition (A2). Summing these inequalities yields that
	\begin{align*}
		\E\Bigg[\frac{1}{N}\sum_{i=1}^N\sup_{t\in[0,T]}|\overleftarrow{Y}^i_t-\overrightarrow{Y}^i_t|^2\Bigg]\leq 12L^2 T^2(1+8a^2)\E\Bigg[\frac{1}{N}\sum_{i=1}^N\sup_{t\in[0,T]}|\overleftarrow{V}^i_t-\overrightarrow{V}^i_t|^2\Bigg].
	\end{align*}
	Choosing $T\leq \varepsilon$ with $\varepsilon>0$ being such that
$12L^2 \varepsilon^2(1+8a^2)\leq 1/2,
$
	then $\Gamma$ is a contraction mapping. Thus, $\Gamma$ has a unique fixed point in $\mathcal{S}^2(\mathbb{F}^{(N)},\mathbb{R})$ when $T\leq \varepsilon$, i.e., there exists a unique solution $\{\overline{Y}^i\}_{1\leq i\leq N}$ solving IPS \eqref{eq7} for some $(\{\overline{Z}^i\}_{1\leq i\leq N},\overline{K}^{(N)})\in (\mathcal{H}^2(\mathbb{F}^{(N)};\,\mathbb{R}^N))^N\times\mathcal{A}_{BV}(\mathbb{F}^{(N)},\mathbb{R})$ on $[0,T]$ when $T$ is small. Since $\{\overline{Y}^i\}_{1\leq i\leq N}$ is unique, applying It\^{o}'s formula, $\{\overline{Z}^i\}_{1\leq i\leq N}$ is unique and consequently $\overline{K}^{(N)}$ is also unique.
	
	For the general $T$, let $n$ be large enough such that $T/n\leq \varepsilon$. For any $k=0,1,\cdots,n$, set $t_k=\frac{kT}{n}$. Let $(\{\overline{Y}^{i,k},\overline{Z}^{i,k}\}_{1\leq i\leq N},\overline{K}^{(N),k})$ be the unique solution on $[t_{k-1},t_k]$ with terminal value $\overline{Y}^{i,k+1}_{t_k}$, driver $f^i$ and obstacles $l,u$, for $k=n,n-1,\ldots,1$, where $\overline{Y}^{i,n+1}_{t_n}=\theta^i$. For any $t\in[t_{k-1},t_k]$ where $k=1,2,\ldots,n$, set 
	\begin{align*}
		\overline{Y}^i_t=\overline{Y}^{i,k}_t,\quad \overline{Z}^i_t=\overline{Z}^{i,k}_t, \quad\text{and}\quad \overline{K}^{(N)}_t=\overline{K}^{(N),k}_t+\sum_{l<k} \overline{K}^{(N),l}_{t_l}.
	\end{align*}
Here, we use the following notation $\sum_{l<1}\overline{K}^{(N),l}_{t_l}=0$.	Then $(\{\overline{Y}^i,\overline{Z}^i\}_{1\leq i\leq N},\overline{K}^{(N)})$ is the solution to the  IPS \eqref{eq7}. The uniqueness follows from the uniqueness on each small time interval. 
\end{proof} 

\begin{proposition}\label{prop4.2'}
	Under Conditions (A1) and (A2), there exists a constant $C$ independent of $N$ such that, for any $1\leq i\leq N$,
	\begin{align*}
		\E\left[\sup_{t\in[0,T]}|\overline{Y}^i_t|^2\right]+\E\left[\int_0^T |\overline{Z}^i_t|^2dt\right]&+\E|\overline{K}^{(N),+}_T|^2+\E|\overline{K}^{(N),-}_T|^2\\
		&\leq C\Bigg(1+\E|\xi|^2+\E\left[\int_0^T|f(t,0)|^2dt\right]\Bigg),
	\end{align*}
 where $(\{\overline{Y}^i,\overline{Z}^i\}_{1\leq i\leq N},\overline{K}^{(N)})$ is the solution to the  IPS \eqref{eq7}.
\end{proposition}

\begin{proof}
	We first prove the estimate for $\overline{Y}^i$. Set 
	\begin{align}
		\label{eqn:barbarU}
		\overline{U}^i_t:=\E\Bigg[\xi^i+\int_t^T f^i(s,\overline{Y}^{i}_s)ds\,\Bigg|\,\mathcal{F}^{(N)}_t\Bigg].
	\end{align}
	Let $\overline{\psi}^{(N)}_t$ and $\overline{\phi}^{(N)}_t$ be $\mathcal{F}^{(N)}_t$-measurable random variables satisfying respectively
	\begin{align}
      \label{eqn:barbarpsi_phi}
		\frac{1}{N}\sum_{i=1}^N h(\overline{U}^i_t+\overline{\psi}^{(N)}_t)=l_t\quad\text{and}\quad \frac{1}{N}\sum_{i=1}^N h(\overline{U}^i_t+ \overline{\phi}^{(N)}_t)=u_t.
	\end{align}
	Since $h$ is linear, it is easy to check that there exists a constant $C$ independent of $N$, such that 
	\begin{equation}\label{barpsin}\begin{split}
			&\E \left[\sup_{s\in[t,T]}|\overline{\psi}^{(N)}_s|^2\right] \leq C\Bigg(1+\E\Bigg[|\xi|^2+\int_0^T |f(s,0)|^2ds\Bigg]+\E\Bigg[\int_t^T\frac{1}{N}\sum_{j=1}^N|Y^j_s|^2ds\Bigg]\Bigg),\\
	&\E \left[\sup_{s\in[t,T]}|\overline{\phi}^{(N)}_s|^2\right] \leq C\Bigg(1+\E\Bigg[|\xi|^2+\int_0^T |f(s,0)|^2ds\Bigg]+\E\Bigg[\int_t^T\frac{1}{N}\sum_{j=1}^N|Y^j_s|^2ds\Bigg]\Bigg).
	\end{split}\end{equation}

We use the following construction as the one in the proof of Theorem \ref{thm5.1}:
	\begin{align*}		\overline{Y}^i_t&=\overline{U}^i_t+\overline{S}_t\\
		&=\overline{U}^i_t+\essinf_{\sigma\in\mathcal{T}^{(N)}_{t,T}}\esssup_{\tau\in\mathcal{T}^{(N)}_{t,T}}\E\Big[(\Psi^{(N)}_T+\Phi^{(N)}_T)\mathbbm{1}_{\{\sigma\wedge \tau=T\}}+\overline{\psi}^{(N)}_\tau \mathbbm{1}_{\{\tau<T,\tau\leq \sigma\}}\\
		&\hspace{7.5cm}+\overline{\phi}^{(N)}_\sigma \mathbbm{1}_{\{\sigma<\tau\}})\,\Big|\,\mathcal{F}^{(N)}_t\Big],
	\end{align*}
	where $\overline{S}$  is the first component of the solution to the following RBSDE:
	\begin{equation}\label{DRBSDE_overline}
		\begin{cases}
			\overline{S}_t=\Psi^{(N)}_T+\Phi^{(N)}_T-\int_t^T \sum_{j=1}^N\overline{Z}^{(N),j}_s dB^j_s+(\overline{K}^{(N)}_T-\overline{K}^{(N)}_t),\vspace{0.2cm}\\
			\overline{\psi}^{(N)}_t\leq \overline{S}_t\leq \overline{\phi}^{(N)}_t, \quad t\in[0,T],\vspace{0.2cm}\\
			\overline{K}^{(N)}=\overline{K}^{(N),+}-\overline{K}^{(N),-} \quad\textrm{ with }\quad \overline{K}^{(N),+},\overline{K}^{(N),-}\in \mathcal{A}^2(\mathbb{F}^{(N)},\mathbb{R}),\vspace{0.2cm}\\
			\int_0^T(\overline{S}_t-\overline{\psi}^{(N)}_t)d\overline{K}^{(N),+}_t=\int_0^T(\overline{\phi}^{(N)}_t-\overline{S}_t)d\overline{K}^{(N),-}_t=0.
		\end{cases}
	\end{equation}
	Noting that $\overline{\psi}^{(N)}_T\leq \Psi^{(N)}_T+\Phi^{(N)}_T\leq \overline{\phi}^{(N)}_T$, it follows that, for any $t\in[0,T]$, 
	\begin{align}\label{yit}
		|\overline{Y}^i_t|\leq &\E\Bigg[|\xi^i|+\int_t^T |f^i(s,0)|ds+L\int_t^T|\overline{Y}^i_s|ds\,\Bigg|\,\mathcal{F}^{(N)}_t\Bigg]\\
		&+\E\Bigg[\sup_{s\in[t,T]}|\overline{\psi}^{(N)}_s|+\sup_{s\in[t,T]}|\overline{\phi}^{(N)}_s|\,\Bigg|\,\mathcal{F}^{(N)}_t\Bigg].\nonumber
	\end{align}
	Applying Doob's inequality and H\"{o}lder's inequality, by equations \eqref{barpsin} and \eqref{yit}, we have
	\begin{align}\label{yit'}
		\E|\overline{Y}^i_t|^2\leq  C\Bigg(1+\E\Bigg[|\xi|^2+\int_0^T |f(s,0)|^2ds\Bigg]+\E\Bigg[\int_t^T\frac{1}{N}\sum_{j=1}^N|\overline{Y}^j_s|^2ds\Bigg]+\E\Bigg[\int_t^T |\overline{Y}^i_s|^2ds\Bigg]\Bigg).
	\end{align}
	Summing over $i$, we obtain that
	\begin{align*}
		\E\left[\frac{1}{N}\sum_{i=1}^N|\overline{Y}^i_t|^2\right]\leq C\Bigg(1+\E\Bigg[|\xi|^2+\int_0^T |f(s,0)|^2ds\Bigg]+\int_t^T\E\Bigg[\frac{1}{N}\sum_{j=1}^N|\overline{Y}^j_s|^2\Bigg]ds\Bigg).
	\end{align*}
	It follows from Gr\"onwall's inequality that 
	\begin{align*}
		\E\left[\frac{1}{N}\sum_{i=1}^N|\overline{Y}^i_t|^2\right]\leq C\Bigg(1+\E\Bigg[|\xi|^2+\int_0^T |f(s,0)|^2ds\Bigg]\Bigg).
	\end{align*}
	Plugging the above inequality into equations \eqref{barpsin} and \eqref{yit'},  using Gr\"onwall's inequality again, we have \begin{equation}\label{barpsinyit}\begin{split}
			\E|\overline{Y}^i_t|^2&\leq C\Bigg(1+\E\Bigg[|\xi|^2+\int_0^T |f(s,0)|^2ds\Bigg]\Bigg),\\			\E\left[\sup_{s\in[0,T]}|\overline{\phi}^{(N)}_s|^2\right]&\leq C\Bigg(1+\E\Bigg[|\xi|^2+\int_0^T |f(s,0)|^2ds\Bigg]\Bigg),\\			\E\left[\sup_{s\in[0,T]}|\overline{\psi}^{(N)}_s|^2\right]&\leq C\Bigg(1+\E\Bigg[|\xi|^2+\int_0^T |f(s,0)|^2ds\Bigg]\Bigg).
	\end{split}\end{equation}
	Finally, combining equations \eqref{yit} and \eqref{barpsinyit}, applying Doob's inequality and H\"{o}lder's inequality, we obtain that
	\begin{equation}\label{estimateyi}
		\E\left[\sup_{t\in[0,T]}|\overline{Y}^i_t|^2\right]
		\leq C\Bigg(1+\E|\xi|^2+\E\left[\int_0^T|f(t,0)|^2dt\right]\Bigg).
	\end{equation}
	
	It remains to prove that  
	\begin{align*}
		\E\left[\int_0^T |\overline{Z}^i_t|^2dt\right]+\E|\overline{K}^{(N),+}_T|^2+\E|\overline{K}^{(N),-}_T|^2
		\leq C\Bigg(1+\E|\xi|^2+\E\left[\int_0^T|f(t,0)|^2dt\right]\Bigg).
	\end{align*}
Similar to the analysis involved in obtaining equation \eqref{k+k-}, by the estimate for $\overline{Y}^i$ in equation \eqref{estimateyi}, there exists a constant $C$ independent of $N$ such that
	\begin{align}\label{estimateKN}
		\E|\overline{K}^{(N),+}_T|^2+\E|\overline{K}^{(N),-}_T|^2
		\leq C\Bigg(1+\E|\xi|^2+\E\left[\int_0^T|f(t,0)|^2dt\right]\Bigg).
	\end{align}
	Applying It\^{o}'s formula to $|\overline{Y}^i_t|^2$, we obtain that 
	\begin{align*}
		|\overline{Y}^i_t|^2+\int_t^T\sum_{j=1}^N |\overline{Z}^{i,j}_s|^2 ds=&|\theta^i|^2+2\int_t^T \overline{Y}^i_s f^i(s,\overline{Y}^i_s)ds+2\int_t^T \overline{Y}^i_sd\overline{K}^{(N)}_s\\
  &-2\int_t^T \overline{Y}^i_s\sum_{j=1}^N \overline{Z}^{i,j}_sdB^j_s.
	\end{align*}
	Letting $t=0$ and taking expectations on both sides yield that 
	\begin{align*}
		\E\left[\int_0^T |\overline{Z}^i_t|^2dt\right]\leq&\E\Bigg[|\theta^i|^2+2\int_0^T \overline{Y}^i_s f^i(s,\overline{Y}^i_s)ds+2\int_0^T \overline{Y}^i_sd\overline{K}^{(N)}_s\Bigg] \\
		\leq &C\Bigg\{\E|\xi^i|^2+\E|\Psi^{(N)}_T|^2+\E|\Phi^{(N)}_T|^2+\E\left[\int_0^T|\overline{Y}^i_t|^2dt\right] \\	
		&\quad+\E\left[\int_0^T|\overline{Y}^i_t ||f^i(t,0)|dt\right]+\E\left[\sup_{t\in[0,T]}|\overline{Y}^i_t|\Big(|\overline{K}^{(N),+}_T|+|\overline{K}^{(N),-}_T|\Big)\right]\Bigg\}\\
		\leq &C\Bigg\{\E|\xi^i|^2+\E|\Psi^{(N)}_T|^2+\E|\Phi^{(N)}_T|^2+\E\left[\sup_{t\in[0,T]}|\overline{Y}^i_t|^2\right]\\
		&\quad+\E\left[\int_0^T|f^i(t,0)|^2dt\right]+\E|\overline{K}^{(N),+}_T|^2+\E|\overline{K}^{(N),-}_T|^2\Bigg\}.
	\end{align*}
	Applying equations \eqref{PsinPhin}, \eqref{estimateyi}, and \eqref{estimateKN}, we obtain the desired result.
\end{proof}

\subsection{Approximation}
\label{sec:linear_Approximation}

In this subsection, we use the IPS \eqref{eq7} to approximate the doubly MRBSDE \eqref{nonlineary}.
For this purpose, consider $\xi$, $f$, $\{\xi^i\}_{1\leq i\leq N}$, $\{f^i\}_{1\leq i\leq N}$ as given in equations \eqref{remarkxi0} and \eqref{remarkxi}. Let $(  Y^i,  Z^i,K)$ be the solution to the following doubly MRBSDE:
\begin{equation}\label{nonlineary_approx}
	\begin{cases}
		Y^i_t=\xi^i+\int_t^T f^i(s,  Y^i_s)ds-\int_t^T   Z^i_s dB^i_s+K_T-K_t, \quad \E[h(Y^i_t)]\in [l_t,u_t],\quad  t\in[0,T], \vspace{0.2cm}\\
			K=K^+ -K^- \quad\textrm{with}\quad  K^+,K^-\in \mathcal{I}([0,T];\,\mathbb{R}),\vspace{0.2cm}\\
		\int_0^T (\E[h(  Y^i_t)]-l_t)dK_t^+=\int_0^T (u_t-\E[h(  Y^i_t)])dK^-_t=0.
	\end{cases}
\end{equation}
Then $(  Y^i,  Z^i,K)$ for $1\leq i\leq N$ are independent copies of $(Y,Z,K)$, which is the unique solution to the doubly MRBSDE \eqref{nonlineary}.

We define
\begin{align*}	  U^i_t:=\E\Bigg[\xi^i+\int_t^Tf^i(s,  Y^i_s)ds \,\Bigg|\,\mathcal{F}^i_t\Bigg].
\end{align*}
Since the Brownian motions $\{B^i\}_{1\leq i\leq N}$ are independent, it is easy to check that  
\begin{align*}
	  U^i_t=\E\Bigg[\xi^i+\int_t^Tf^i(s,  Y^i_s)ds \,\Bigg|\,\mathcal{F}^{(N)}_t\Bigg].
\end{align*}
Define 
\begin{align*}
	\Psi_T:=\inf\Big\{x\geq 0;\,\E[h(\xi^i+x)]\geq l_T\Big\}\quad\text{and}\quad
	 \Phi_T:=\sup\Big\{x\leq 0;\,\E[h(\xi^i+x)]\leq u_T\Big\}.
\end{align*}
Let $\{  \psi_t\}_{t\in[0,T]}$ and $\{  \phi_t\}_{t\in[0,T]}$ be two deterministic functions satisfying respectively
\begin{align*}
\E[h(  U^i_t+  \psi_t)]=l_t\quad\text{and}\quad \E[h(  U^i_t+  \phi_t)]=u_t.
\end{align*}
It is clear that $$  \psi_T\leq \Psi_T=0=\Phi_T\leq   \phi_T.$$ Thus, we may write
\begin{align}\label{psi+phi-}
	\Psi_T=  (\psi_T)^+\quad\text{and}\quad \Phi_T=-  (\phi_T)^-.
\end{align}
Next, by Theorem 3.6 in \cite{falkowski2022backward}, we have
\begin{align*}
	K_T-K_t&=\E[  Y^i_t]-\E[  U^i_t]=\inf_{s\in[t,T]}\sup_{q\in[t,T]}\E\Big[  \psi_q \mathbbm{1}_{\{q<T,q\leq s\}}+  \phi_s \mathbbm{1}_{\{s<q\}}\Big]\\
	&=\essinf_{\sigma\in \mathcal{T}^{(N)}_{t,T}}\esssup_{\tau\in\mathcal{T}^{(N)}_{t,T}}\E\Big[(\Psi_\sigma+\Phi_\tau)\mathbbm{1}_{\{\sigma\wedge \tau=T\}}+  \psi_{\sigma} \mathbbm{1}_{\{\sigma<T,\sigma\leq \tau\}}+  \phi_\tau \mathbbm{1}_{\{\tau<\sigma\}}\,\Big|\,\mathcal{F}^{(N)}_t\Big].
\end{align*}
It follows that 
\begin{align}\label{baryi}
	  Y^i_t=  U^i_t+\essinf_{\sigma\in \mathcal{T}^{(N)}_{t,T}}\esssup_{\tau\in\mathcal{T}^{(N)}_{t,T}}\E\Big[(\Psi_\sigma+\Phi_\tau)\mathbbm{1}_{\{\sigma\wedge \tau=T\}}+  \psi_{\sigma} \mathbbm{1}_{\{\sigma<T,\sigma\leq \tau\}}+  \phi_\tau \mathbbm{1}_{\{\tau<\sigma\}}\,\Big|\,\mathcal{F}^{(N)}_t\Big].
\end{align}

\begin{theorem}\label{thm4.3}
 Consider $(\{\overline{Y}^i,\overline{Z}^i\}_{1\leq i\leq N},K^{(N)})$ as the solution to the IPS \eqref{eq7}  and $(\{  Y^i,   Z^i\}_{1\leq i\leq N}, K)$ as the solution to the doubly MRBSDE \eqref{nonlineary_approx}.  
     Under Conditions (A1) and (A2), if  $\sup_{t\in[0,T]}\E|  Z^1_t|^4<\infty$, there exists a constant $C$ independent of $N$, such that
		\begin{align*}
			\E\Bigg[\sup_{t\in[0,T]}|\overline{Y}^i_t-  Y^i_t|^2\Bigg]\leq CN^{-1}, \quad \quad\E\Bigg[\sup_{t\in[0,T]}|K^{(N)}_t-K_t|^2\Bigg]\leq CN^{-1/2} \\
   \text{and}\quad\quad \E\Bigg[\int_0^T|\overline{Z}^i_t-  Z^i_t e_i|^2dt\Bigg]\leq CN^{-1/2},
		\end{align*}
where $(e_1,\cdots,e_N)$ is the canonical basis in $\mathbb{R}^N$ and thus
$$\overline{Z}^i_t-  Z^i_t e_i=\Big(\overline{Z}^{i,1}_t,\cdots, \overline{Z}^{i,i-1}_t,\overline{Z}^{i,i}_t-  Z^{i}_t,\overline{Z}^{i,i+1}_t,\cdots, \overline{Z}^{i,N}_t\Big).$$

\end{theorem}

Before proving Theorem \ref{thm4.3}, we need to first prove a technical lemma, which will also be needed in the next section. 
	Let  $  \psi^{(N)}$ and $  \phi^{(N)}$ be processes satisfying respectively 
	\begin{align}\label{beforeconvergepsi}
		&\frac{1}{N}\sum_{j=1}^N h(  \psi^{(N)}_t+  U^j_t)=l_t\quad\text{and} \quad \frac{1}{N}\sum_{j=1}^N h(  \phi^{(N)}_t+  U^j_t)=u_t.
	\end{align}
 \begin{lemma}\label{convergepsi} Under Conditions (B1) and (B2), if $\sup_{t\in[0,T]}\E |  Z^1_t|^4<\infty$ and $h$ is twice continuously differentiable with bounded derivatives, then there exists a constant $C$ independent of $N$, such that 
         \begin{align*}
             \E\Bigg[\sup_{s\in[0,T]}\Big|\psi_s-  \psi^{(N)}_s\Big|^2\Bigg]\leq CN^{-1} \quad\text{and} \quad\E\Bigg[\sup_{s\in[0,T]}\Big|\phi_s-  \phi^{(N)}_s\Big|^2\Bigg]\leq CN^{-1}.
         \end{align*}
 \end{lemma}

 \begin{proof}
     It suffices to prove the convergence rate for $\psi-  \psi^{(N)}$, since that for $\phi-  \phi^{(N)}$ can be derived analogously. For any $t\in[0,T]$, let $\nu_t$ be the common law of the random variables $\{  U^i_t\}_{1\leq i\leq N}$ and let $\nu^{(N)}_t$ be their empirical law, i.e.,
     \begin{align*}
         \nu_t^{(N)}:=\frac{1}{N}\sum_{i=1}^N \delta_{  U^i_t},
     \end{align*}
     where $\delta(\cdot)$ is the 
     Dirac delta function. 
     For any $\mu\in \mathcal{P}_1(\mathbb{R})$, where $\mathcal{P}_1(\mathbb{R})$ is the set of probability measures on $\mathbb{R}$ with finite first moment, and for any  $x\in\mathbb{R}$, we define
     \begin{align*}
           H(x,\mu):=\int h(x+y)\mu(dy).
     \end{align*}
By the definitions of $\psi_t$ and $  \psi^{(N)}_t$, we have
     \begin{align*}
           H(  \psi^{(N)}_t,\nu^{(N)}_t)=  H(\psi_t,\nu_t)=l_t.
     \end{align*}
     By Lemma 2.1 in \cite{Briand2020Particles}, for any $\mu\in\mathcal{P}_1(\mathbb{R})$ and $x,y\in\mathbb{R}$, we have
     \begin{align*}
         \gamma_l|x-y|\leq |  H(x,\mu)-  H(y,\mu)|\leq \gamma_u|x-y|.
     \end{align*}
     All the above analysis indicates that
     \begin{align*}
         \Big|  \psi^{(N)}_t-\psi_t\Big|\leq \frac{1}{\gamma_l}\Big|  H(  \psi^{(N)}_t,\nu_t^{(N)})-  H(\psi_t,\nu^{(N)}_t)\Big|\leq \frac{1}{\gamma_l}\Big|  H(\psi_t,\nu_t)-  H(\psi_t,\nu^{(N)}_t)\Big|.
     \end{align*}
     It follows that 
     \begin{equation}\label{eq10}
         \E\Bigg[\sup_{s\in[0,T]}\Big|\psi_s-  \psi^{(N)}_s\Big|^2\Bigg]\leq \frac{1}{\gamma_l^2}\E\Bigg[\sup_{t\in[0,T]}\Big|  H(\psi_t,\nu_t)-  H(\psi_t,\nu^{(N)}_t)\Big|^2\Bigg].
     \end{equation}
 
 By the proof of Proposition 2.7 in \cite{Briand2020Particles} or the proof of Theorem 4.3 in \cite{briand2021particles}, $t\rightarrow \psi_t$ is locally Lipschitz. Let $\psi'$ be the derivative of $\psi$. We define  
     \begin{align*}
           H(\psi_t,\nu^{(N)}_t)-  H(\psi_t,\nu_t)=\frac{1}{N}\sum_{i=1}^N\Big(h(V^i_t)-\E h(V^i_t)\Big),
     \end{align*}
     where $V^i_t:=\psi_t+  U^i_t.$
      Applying It\^{o}'s formula, we obtain that 
     \begin{align*}
         H(\psi_t,\nu^{(N)}_t)-  H(\psi_t,\nu_t)
         &=\frac{1}{N}\sum_{i=1}^N\big(h(V^i_T)-\E h(V^i_T)\big)-\int_t^T\frac{1}{N}\sum_{i=1}^Nh'(V^i_t)  Z^i_s dB^i_s\\
         &\quad-\int_t^T \frac{1}{N}\sum_{i=1}^N\big( \eta^i_s-\E  \eta^i_s \big)ds-\frac{1}{2}\int_t^T \frac{1}{N}\sum_{i=1}^N\big(\zeta^i_s-\E\zeta^i_s\big)ds,
     \end{align*}
where
     \begin{align*}
         \eta^i_s=h'(V^i_s)(\psi'_s-f^i(s,  Y^i_s))\quad\text{and} \quad \zeta^i_s=h''(V^i_s)|  Z^i_s|^2.
     \end{align*}
     Applying Doob's inequality and H\"{o}lder's inequality, we obtain that
     \begin{align*}
         &\hspace{-1cm}\E\Bigg[\sup_{t\in[0,T]}\Big|  H(\psi_t,\nu_t)-  H(\psi_t,\nu^{(N)}_t)\Big|^2\Bigg]\\
         \leq &\frac{4}{N^2}\Bigg\{\E\Bigg(\sum_{i=1}^N\Big(h(V^i_T)-\E [h(V^i_T)]\Big)\Bigg)^2+T\int_0^T \E \Bigg(\sum_{i=1}^N\Big(\eta^i_s-\E[\eta^i_s]\Big)\Bigg)^2 ds\\
&\hspace{2cm}+4\int_0^T\sum_{i=1}^N\E\Big[|h'(V^i_t)  Z^i_s|^2\Big] ds+\frac{T}{4}\int_t^T \E \Bigg(\sum_{i=1}^N\Big(\zeta^i_s-\E \zeta^i_s\Big)\Bigg)^2 ds\Bigg\}\\
         \leq 
         &\frac{C}{N}\Bigg\{\operatorname{Var}[h(V^1_T)]+\int_0^T \operatorname{Var}[\eta^1_s]ds+\int_0^T \E\Big|h'(V^1_t)  Z^1_s\Big|^2ds+\int_0^T \operatorname{Var}[\zeta^1_s]ds\Bigg\}\\
         \leq &\frac{C}{N}\Bigg\{\E|V^1_T|^2+\int_0^T\E|f^1(s,  Y_s^1)|^2ds+\int_0^T \E|  Z^1_s|^2ds+\int_0^T \E|  Z^1_s|^4ds\Bigg\}\\
         \leq&\frac{C}{N}\Bigg\{\E[\xi^2]+|\psi_T|^2+\E\Bigg[\int_0^T|f(t,0)|^2dt\Bigg]+\E\Bigg[\sup_{t\in[0,T]}|  Y^1_t|^2\Bigg]+\int_0^T \E|  Z^1_s|^4ds\Bigg\}.
     \end{align*}
      Plugging the above estimate into equation \eqref{eq10}, we obtain the desired result. 
 \end{proof}

\begin{proof}[Proof of Theorem \ref{thm4.3}]
	We proceed with the proof in the following two steps:
 \medskip
	
	\noindent\textbf{Step 1.} For the convergence rate of $K$, we claim that 
	\begin{align*}
		\E\Bigg[\sup_{t\in[0,T]}|K^{(N)}_t-K_t|^2\Bigg]\leq C\Bigg(\E\Bigg[\sup_{t\in[0,T]}|\overline{Y}^i_t-  Y^i_t|^2\Bigg]\Bigg)^{1/2},
  \end{align*}
 and for the convergence rate of $Z$, we claim that \begin{align*}\E\Bigg[\int_0^T|\overline{Z}^i_t-  Z^i_t e_i|^2dt\Bigg]\leq C\Bigg(\E\Bigg[\sup_{t\in[0,T]}|\overline{Y}^i_t-  Y^i_t|^2\Bigg]\Bigg)^{1/2}.
	\end{align*}
	In fact, noting that by equation \eqref{eqn:theta_def} and the IPS \eqref{eq7} we have $\overline{Y}^i_T=\theta^i=\xi^i+\Psi^{(N)}_T+\Phi^{(N)}_T$; by equation \eqref{nonlineary_approx}  we have $  Y^i_T=\xi^i$. Hence, $\Psi^{(N)}_T+\Phi^{(N)}_T=\overline{Y}^i_t-  Y^i_t$. Applying It\^{o}'s formula to $|\overline{Y}^i_t-  Y^i_t|^2$ and taking expectations yield that 
	\begin{align*}		
		&\hspace{-0.4cm}\E\Bigg[\int_0^T|\overline{Z}^i_t-  Z^i_t e_i|^2dt\Bigg]\\
  \leq& \E\Bigg[|\Psi^{(N)}_T+\Phi^{(N)}_T|^2+2\int_0^T
		|\overline{Y}^i_s-  Y^i_s|\cdot |f^i(s,\overline{Y}_s^i)-f^i(s,  Y_s^i)|ds\\
  &\hspace{6cm}+2\int_0^T|\overline{Y}^i_s-  Y^i_s|(dK^{(N)}_s-dK_s)\Bigg]\\
		\leq& C\Bigg\{\E|\Psi^{(N)}_T+\Phi^{(N)}_T|^2+\E\Big[\sup_{t\in[0,T]}|\overline{Y}^i_t-  Y^i_t|^2\Big]
		\\
	&\qquad+\Big(\E\Big[\sup_{t\in[0,T]}|\overline{Y}^i_t-  Y^i_t|^2\Big]\Big)^{1/2}\Big(\E|K^{(N),+}_T|^2+\E|K^{(N),-}_T|^2+\E|K^+_T|^2+\E|K^-_T|^2\Big)^{1/2}\Bigg\}.
	\end{align*}
 Then by Proposition \ref{prop4.2'}, 
	\begin{align*}
		\E\Bigg[\int_0^T|\overline{Z}^i_t-  Z^i_t e_i|^2dt\Bigg]\leq C\Bigg(\E\Bigg[\sup_{t\in[0,T]}|\overline{Y}^i_t-  Y^i_t|^2\Bigg]\Bigg)^{1/2}.
	\end{align*}
	Set $  Z^{i,j}\equiv 0$ for $j\neq i$ and $  Z^{i,i}\equiv   Z^i$. Then we can write 
	\begin{align*}
		K^{(N)}_t-K_t=&\overline{Y}^i_0-  Y^i_0-\overline{Y}^i_t+  Y^i_t-\int_0^t \Big(f^i(s,\overline{Y}^i_s)-f^i(s,  Y^i_s)\Big)ds\\
		&+\int_0^t \sum_{j=1}^N\Big(\overline{Z}^{i,j}_s-  Z^{i,j}_s\Big)dB^j_s.
	\end{align*}
	It is easy to check that  
	\begin{align*}
		\E\Bigg[\sup_{t\in[0,T]}|K^{(N)}_t-K_t|^2\Bigg]&\leq C\Bigg(\E\Bigg[\sup_{t\in[0,T]}|\overline{Y}^i_t-  Y^i_t|^2\Bigg]+\E\Bigg[\int_0^T\sum_{j=1}^N|\overline{Z}^{i,j}_t-  Z^{i,j}_t|^2dt\Bigg] \Bigg)\\
        &=C\Bigg(\E\Bigg[\sup_{t\in[0,T]}|\overline{Y}^i_t-  Y^i_t|^2\Bigg]+\E\Bigg[\int_0^T|\overline{Z}^i_t-  Z^i_t e_i|^2dt\Bigg] \Bigg) ,
	\end{align*}
	as desired. 
	\bigskip 
	
	\noindent\textbf{Step 2.} It remains to prove the convergence rate for the $Y$-term.  Set 
	\begin{align*}
		\overline{U}^i_t=\E\Bigg[\xi^i+\int_t^Tf^i(s,\overline{Y}^i_s)ds \,\Bigg|\,\mathcal{F}^{(N)}_t\Bigg].
	\end{align*}
	Let $\overline{\psi}^{(N)}$ and $\overline{\phi}^{(N)}$
be processes satisfying respectively
	\begin{align*}
		\frac{1}{N}\sum_{j=1}^N h(\overline{\psi}^{(N)}_t+\overline{U}^j_t)=l_t\quad\text{and} \quad \frac{1}{N}\sum_{j=1}^N h(\overline{\phi}^{(N)}_t+\overline{U}^j_t)=u_t.
	\end{align*}
Similar to the proof of Proposition \ref{prop4.2'},	recalling the definitions of $\Psi^{(N)}_T$ and $\Phi^{(N)}_T$ in equation \eqref{PsinPhin'} and the definitions of $\psi^{(N)}$ and $\phi^{(N)}$ in \eqref{beforeconvergepsi}, it is easy to check that 
	\begin{align}\label{Psi+Phi-}
		\Psi^{(N)}_T=( {\psi}^{(N)}_T)^+\quad\text{and}\quad \Phi^{(N)}_T=-({\phi}^{(N)}_T)^-.
	\end{align}
	Equations \eqref{psi+phi-} and \eqref{Psi+Phi-} imply that 
	\begin{align}\label{pp}
		|\Phi_T-\Phi^{(N)}_T|\leq |\phi_T- {\phi}^{(N)}_T|\quad\text{and}\quad |\Psi_T-\Psi^{(N)}_T|\leq |\psi_T- {\psi}^{(N)}_T|.
	\end{align}
	Similar to the proof of Proposition \ref{prop4.2'}, for any $1\leq i\leq N$, we have
	\begin{align*}
		\overline{Y}^i_t=\overline{U}^i_t+\essinf_{\sigma\in \mathcal{T}^{(N)}_{t,T}}\esssup_{\tau\in\mathcal{T}^{(N)}_{t,T}}\E\Big[R^{(N)}(\sigma,\tau)\,\Big|\,\mathcal{F}^{(N)}_t\Big],
	\end{align*}
	where 
	\begin{equation}
		{R}^{(N)}(\sigma,\tau):=(\Psi^{(N)}_T+\Phi^{(N)}_T)\mathbbm{1}_{\{\sigma\wedge \tau=T\}}+\overline{\psi}^{(N)}_\tau \mathbbm{1}_{\{\tau<T,\tau\leq \sigma\}}+\overline{\phi}^{(N)}_\sigma \mathbbm{1}_{\{\sigma<\tau\}}.
	\end{equation}
 It together with equations \eqref{baryi} and \eqref{pp} gives the following inequality: for any $t\in[r,T]$ with $r$ being a constant in $[0,T]$,
	\begin{align*}
		|\overline{Y}^i_t-  Y^i_t|\leq &|\overline{U}^i_t-  U^i_t|+\E\Big[|\Phi_T-\Phi^{(N)}_T|\,\Big|\,\mathcal{F}^{(N)}_t\Big]+\E\Big[|\Psi_T-\Psi^{(N)}_T|\,\Big|\,\mathcal{F}^{(N)}_t\Big]\\
		&+\E\Bigg[\sup_{s\in[t,T]}|\phi_s-\overline{\phi}^{(N)}_s|\,\Bigg|\,\mathcal{F}^{(N)}_t\Bigg]+\E\Bigg[\sup_{s\in[t,T]}|\psi_s-\overline{\psi}^{(N)}_s|\,\Bigg|\,\mathcal{F}^{(N)}_t\Bigg]\\
		\leq &L\E\Bigg[\int_r^T|\overline{Y}^i_s-  Y^i_s|ds\,\Bigg|\,\mathcal{F}^{(N)}_t\Bigg]+\E\Bigg[\sup_{s\in[r,T]}|  \overline{\phi}^{(N)}_s-{\phi}^{(N)}_s|\,\Bigg|\,\mathcal{F}^{(N)}_t\Bigg]\\
  &+\E\Bigg[\sup_{s\in[r,T]}|  \overline{\psi}^{(N)}_s-{\psi}^{(N)}_s|\,\Bigg|\,\mathcal{F}^{(N)}_t\Bigg]+2\E\Bigg[\sup_{s\in[r,T]}|\phi_s-  {\phi}^{(N)}_s|\,\Bigg|\,\mathcal{F}^{(N)}_t\Bigg]\\
  &+2\E\Bigg[\sup_{s\in[r,T]}|\psi_s-  {\psi}^{(N)}_s|\,\Bigg|\,\mathcal{F}^{(N)}_t\Bigg].
	\end{align*}
By Doob's inequality and H\"{o}lder's inequality, there exists a constant $C$ independent of $N$, such that
	\begin{align}\label{hatyi'}
		&\hspace{-0.4cm}\E \Bigg[\sup_{t\in[r,T]}|\overline{Y}^i_t-  Y^i_t|^2 \Bigg]\\
  \leq & 
  C\Bigg\{\E\Bigg[\int_r^T|\overline{Y}^i_s-  Y^i_s|^2ds\Bigg] +\E \Bigg[\sup_{s\in[r,T]}|\overline{\phi}^{(N)}_s-  \phi^{(N)}_s|^2\Bigg] +\E \Bigg[\sup_{s\in[r,T]}|\overline{\psi}^{(N)}_s-  \psi^{(N)}_s|^2\Bigg] \nonumber\\
			&\hspace{3.8cm}+\E \Bigg[\sup_{s\in[0,T]}|\phi_s- {\phi}^{(N)}_s|^2\Bigg]+\E \Bigg[\sup_{s\in[0,T]}|\psi_s-  {\psi}^{(N)}_s|^2 \Bigg]\Bigg\}.\nonumber
	\end{align} 
	By equation \eqref{eqn:Lipschitz}, for any $t\in[r,T]$,  we have
	\begin{align*}
		&|\overline{\phi}^{(N)}_t-  \phi^{(N)}_t|\leq \frac{a}{N}\sum_{j=1}^N |\overline{U}^j_t-  U^j_t|\leq La\E\Bigg[\frac{1}{N}\sum_{j=1}^N\int_r^T |\overline{Y}^j_u-  Y^j_u|du\,\Bigg|\,\mathcal{F}^{(N)}_t\Bigg],\\
		&|\overline{\psi}^{(N)}_t- \psi^{(N)}_t|\leq \frac{a}{N}\sum_{j=1}^N |\overline{U}^j_t-  U^j_t|\leq La\E\Bigg[\frac{1}{N}\sum_{j=1}^N\int_r^T |\overline{Y}^j_u-  Y^j_u|du\,\Bigg|\,\mathcal{F}^{(N)}_t\Bigg].
	\end{align*}
	Plugging the above inequality into equation \eqref{hatyi'}, Doob's inequality and H\"{o}lder's inequality imply that for any $r\in[0,T]$, 
	\begin{align}\label{eq9'}
		\E\Bigg[\sup_{t\in[r,T]}|\overline{Y}^i_t-  Y^i_t|^2\Bigg]\leq C\Bigg\{\E\Bigg[\int_r^T|\overline{Y}^i_s-  Y^i_s|^2ds\Bigg]+\E\Bigg[\frac{1}{N}\sum_{j=1}^N\int_r^T |\overline{Y}^j_s-  Y^j_s|^2ds\Bigg]\Bigg\}+C J^{(N)},
	\end{align}
	where 
	\begin{align*}
		J^{(N)}:=\E \Bigg[\sup_{s\in[0,T]}|\phi_s-  \phi^{(N)}_s|^2\Bigg]+\E\Bigg[\sup_{s\in[0,T]}|\psi_s-  \psi^{(N)}_s|^2.\Bigg]
	\end{align*}
	It follows that 
	\begin{align*}
		\E\Bigg[\frac{1}{N}\sum_{i=1}^N\sup_{t\in[r,T]}|\overline{Y}^i_t-  Y^i_t|^2\Bigg]\leq &C\E\Bigg[\frac{1}{N}\sum_{i=1}^N\int_r^T |\overline{Y}^i_s-  Y^i_s|^2ds\Bigg]+C J^{(N)}\\
		\leq &C\int_r^T\E\Bigg[ \frac{1}{N}\sum_{i=1}^N\sup_{s\in[u,T]}|\overline{Y}^i_s-  Y^i_s|^2\Bigg]du+C J^{(N)},
	\end{align*}
	and then by Gr\"onwall's inequality that
	\begin{align*}
		\E\Bigg[\frac{1}{N}\sum_{i=1}^N\sup_{t\in[0,T]}|\overline{Y}^i_t-  Y^i_t|^2\Bigg]\leq C J^{(N)}.
	\end{align*}
	Combining the above estimate with equation \eqref{eq9'}, we finally obtain that
	\begin{align*}
		\E\Bigg[\sup_{t\in[0,T]}|\overline{Y}^i_t-  Y^i_t|^2\Bigg]\leq CJ^{(N)},
	\end{align*} 
	where $C$ is a constant only depending on $a$, $L$, and $T$. Applying Lemma \ref{convergepsi}, we obtain the desired result.
\end{proof}

\section{The case of nonlinear reflection}
\label{sec:nonlinear_reflection}
In this section, we examine the case of nonlinear reflection. 
 Similar to the previous section, we commence by establishing the well-posedness of a simplified IPS \eqref{eq7_2_s_tilde} in Section \ref{sec:nonlinear_constant}. Building upon this foundation, we  investigate the well-posedness of the IPS  \eqref{eq7_2_tilde} in Section \ref{sec:nonlinear_nonconstant}, and then we attain the POC in Section \ref{sec:nonlinear_Approximation}.

\subsection{The particle system with constant driver}
\label{sec:nonlinear_constant}
We aim to approximate the solution to the doubly MRBSDE \eqref{nonlinearyz} recalled here as follows:
\begin{equation*}
	\begin{cases}
		Y_t=\xi+\int_t^T f(s,  Y_s)ds-(M_T-M_t)+K_T-K_t, \quad \E[h(Y_t)]\in [l_t,u_t],\quad  t\in[0,T], \vspace{0.2cm}\\
		K=K^+ -K^- \quad\textrm{with}\quad  K^+,K^-\in \mathcal{I}([0,T];\,\mathbb{R}),\vspace{0.2cm}\\
		\int_0^T (\E[h(  Y_t)]-l_t)dK_t^+=\int_0^T (u_t-\E[h(  Y_t)])dK^-_t=0,
	\end{cases}
\end{equation*}
 using the IPS \eqref{eq7_2_tilde} recalled here as follows: for any $1\leq i\leq N$,
\begin{equation*}
	\begin{cases}
		\widetilde{Y}^i_t=\E\Big[\xi^i+\int_t^T f^i(s,\widetilde{Y}^i_s) ds\,\Big|\,\mathcal{F}^{(N)}_t\Big]+(\Psi^{(N)}_T+\Phi^{(N)}_T)-(\widetilde{M}^{(N)}_T-\widetilde{M}^{(N)}_t)+\widetilde{K}^{(N)}_T-\widetilde{K}^{(N)}_t,\vspace{0.2cm}\\
		l_t\leq \frac{1}{N}\sum_{i=1}^N h(\widetilde{Y}^i_t)\leq u_t, \quad t\in[0,T],\vspace{0.2cm}\\
		\widetilde{K}^{(N)}=\widetilde{K}^{(N),+}-\widetilde{K}^{(N),-}\quad\textrm{with}\quad \widetilde{K}^{(N),+},\widetilde{K}^{(N),-}\in \mathcal{A}(\mathbb{F}^{(N)};\mathbb{R}),\vspace{0.2cm}\\
		\int_0^T(\frac{1}{N}\sum_{i=1}^N h(\widetilde{Y}^i_t)-l_t)d\widetilde{K}^{(N),+}_t=\int_0^T (u_t-\frac{1}{N}\sum_{i=1}^N h(\widetilde{Y}^i_t))d\widetilde{K}^{(N),-}_t=0,
	\end{cases}
\end{equation*}
where $\Psi^{(N)}_T$, $\Phi^{(N)}_T$ are defined in equation \eqref{PsinPhin'}. The IPS \eqref{eq7_2_tilde} is a multi-dimensional doubly RBSDE, whose solution is a family of processes $(\{\widetilde{Y}^i\}_{1\leq i\leq N},\widetilde{M}^{(N)},\widetilde{K}^{(N)})$ taking values in $\mathbb{R}^N\times \mathbb{R}\times \mathbb{R}$. To establish its well-posedness, we need to first study its special case that the driver $f$ does not depend on $y$, i.e.,
\begin{equation}\label{eq7_2_s_tilde}
	\begin{cases}
		\tilde{Y}^i_t=\E\Big[\xi^i+\int_t^T f^i(s) ds\,\Big|\,\mathcal{F}^{(N)}_t\Big]+(\Psi^{(N)}_T+\Phi^{(N)}_T)-(\tilde{M}^{(N)}_T-\tilde{M}^{(N)}_t)+\tilde{K}^{(N)}_T-\tilde{K}^{(N)}_t,\vspace{0.2cm}\\
		l_t\leq \frac{1}{N}\sum_{i=1}^N h(\tilde{Y}^i_t)\leq u_t, \quad t\in[0,T],\vspace{0.2cm}\\
		\tilde{K}^{(N)}=\tilde{K}^{(N),+}-\tilde{K}^{(N),-}\quad\textrm{with}\quad \tilde{K}^{(N),+},\tilde{K}^{(N),-}\in \mathcal{A}(\mathbb{F}^{(N)};\mathbb{R}),\vspace{0.2cm}\\
		\int_0^T(\frac{1}{N}\sum_{i=1}^N h(\tilde{Y}^i_t)-l_t)d\tilde{K}^{(N),+}_t=\int_0^T (u_t-\frac{1}{N}\sum_{i=1}^N h(\tilde{Y}^i_t))d\tilde{K}^{(N),-}_t=0.
	\end{cases}
\end{equation}

Define $\tilde{U}^{i}_t$ as
\begin{align}
	\tilde{U}^{i}_t:=\E\Bigg[\xi^i+\int_t^T f^i(s) ds\,\Bigg|\, \mathcal{F}^{(N)}_t\Bigg],
\end{align}
and define $\tilde{\psi}^{(N)}_t$ and $\tilde{\phi}^{(N)}_t$ as $\mathcal{F}^{(N)}_t$-measurable random variables satisfying respectively
\begin{align*}
	\frac{1}{N}\sum_{i=1}^N h(\tilde{U}^{i}_t+\tilde{\psi}^{(N)}_t)=l_t\quad\text{and}\quad \frac{1}{N}\sum_{i=1}^N h(\tilde{U}^{i}_t+\tilde{\phi}^{(N)}_t)=u_t.
\end{align*}
Readers may have already noticed that the definitions of $\tilde{U}^{i}_t$, $\tilde{\psi}^{(N)}_t$, and $\tilde{\phi}^{(N)}_t$ are presented in the same mathematical formula as those of $\bar{U}^{i}_t$, $\bar{\psi}^{(N)}_t$, and $\bar{\phi}^{(N)}_t$ in Section \ref{sec:linear_constant}. 
However, here we employ distinct notations primarily for the sake of consistency of the notations in this section, and secondarily because the function $h$ have different regularities in these two sections.
In Section \ref{sec:nonlinear_constant}, explicit expressions for $\bar{\psi}^{(N)}$ and $\bar{\phi}^{(N)}$ can be derived using equation \eqref{phinpsin}, thanks to the linearity of the loss function $h$. However, in the case of a nonlinear loss function, as is the focus here, we need to obtain the following estimation.
\begin{lemma}\label{lempsiphi}
Suppose Conditions (B1) and (B2) hold. The processes $\tilde{\psi}^{(N)}$ and $\tilde{\phi}^{(N)}$ belong to $\mathcal{S}^2(\mathbb{F}^{(N)};\mathbb{R})$ satisfying \begin{align}\label{separate1}
\tilde{\psi}^{(N)}_T\leq \Psi^{(N)}_T+{\Phi}^{(N)}_T\leq \tilde{\phi}^{(N)}_T\end{align}
and
\begin{align}\label{separate}
\inf_{t\in[0,T]}(\tilde{\phi}^{(N)}_t-\tilde{\psi}^{(N)}_t)>0.
\end{align}
Specifically, we have
\begin{equation}\label{psinphin}\begin{split}
&\E\Bigg[\sup_{t\in[0,T]}|\tilde{\psi}^{(N)}_t|^2\Bigg]\leq C\Bigg(1+\E\Bigg[|\xi|^2+\int_0^T |f(s)|^2ds\Bigg]\Bigg),\\
&\E\Bigg[\sup_{t\in[0,T]}|\tilde{\phi}^{(N)}_t|^2\Bigg]\leq C\Bigg(1+\E\Bigg[|\xi|^2+\int_0^T |f(s)|^2ds\Bigg]\Bigg).
\end{split}\end{equation}
\end{lemma}

\begin{proof}
We first prove equation \eqref{separate1}. It is obvious that  $\tilde{\phi}^{(N)}_t\geq \tilde{\psi}^{(N)}_t$ for $t\in[0,T]$. Recalling that $\Psi^{(N)}_T\Phi^{(N)}_T=0$, we prove the result for the following 3 cases:
when $\Psi^{(N)}_T>0$, we have
\begin{align*}
\frac{1}{N}\sum_{i=1}^N h(\xi^i+\Psi^{(N)}_T)= l_T<u_T,
\end{align*}
which implies that $\Psi^{(N)}_T=\tilde{\psi}^{(N)}_T$ and $\Phi^{(N)}_T=0\leq \tilde{\phi}^{(N)}_T$; 
when  $\Psi^{(N)}_T=\Phi^{(N)}_T=0$, we have
\begin{align*}
l_T\leq \frac{1}{N}\sum_{i=1}^N h(\xi^i)\leq u_T,
\end{align*}
which gives $\tilde{\psi}^{(N)}_T\leq 0\leq \tilde{\phi}^{(N)}_T$; 
when $\Phi^{(N)}_T<0$, we have
\begin{align*}
\frac{1}{N}\sum_{i=1}^N h(\xi^i+\Phi^{(N)}_T)= u_T>l_T,
\end{align*}
which implies that $\Phi^{(N)}_T=\tilde{\phi}^{(N)}_T$ and $\Psi^{(N)}_T=0\geq \tilde{\psi}^{(N)}_T$. By the above analysis, we obtain the desired result.

Then we show the estimates in equation \eqref{psinphin} hold. For this purpose, given $\breve{U}^i\in\mathcal{S}^2(\mathbb{F}^{(N)};\mathbb{R})$ for $1\leq i\leq N$, let $\breve{\psi}^{(N)}$ and $\breve{\phi}^{(N)}$ be progressively measurable processes such that, for $t\in[0,T]$,
\begin{align*}
\frac{1}{N}\sum_{i=1}^N h(\breve{U}^i_t+\breve{\psi}^{(N)}_t)=l_t\quad\text{and}\quad\frac{1}{N}\sum_{i=1}^N h(\breve{U}^i_t+\breve{\phi}^{(N)}_t)=u_t.
\end{align*}
We first claim that 
\begin{align}\label{eq4}
|\tilde{\psi}^{(N)}_t-\breve{\psi}^{(N)}_t|\leq \frac{\gamma_u}{\gamma_l N}\sum_{j=1}^N|\tilde{U}^j_t-\breve{U}^j_t|\quad\text{and}\quad |\tilde{\phi}^{(N)}_t-\breve{\phi}^{(N)}_t|\leq \frac{\gamma_u}{\gamma_l N}\sum_{j=1}^N|\tilde{U}^j_t-\breve{U}^j_t|.
\end{align}
We only prove the first inequality in equation \eqref{eq4} and the second inequality can be obtained analogously. In fact, noting that $h$ is bi-Lipschitz and increasing, for any $1\leq i\leq N$, we have
\begin{align*}
h\Bigg(\tilde{\psi}^{(N)}_t+\frac{\gamma_u}{\gamma_l N}\sum_{j=1}^N|\tilde{U}^j_t-\breve{U}^j_t|+\breve{U}^i_t\Bigg)\geq &\gamma_l\frac{\gamma_u}{\gamma_l N}\sum_{j=1}^N|\tilde{U}^j_t-\breve{U}^j_t|+h(\tilde{\psi}^{(N)}_t+\breve{U}^i_t)\\
\geq &\frac{\gamma_u}{N}\sum_{j=1}^N|\tilde{U}^j_t-\breve{U}^j_t|+h(\tilde{\psi}^{(N)}_t+\tilde{U}^{i}_t)-\gamma_u|\tilde{U}^{i}_t-\breve{U}^i_t|.
\end{align*}
Summing over $i$, by the definition of $\tilde{\psi}^{(N)}$, we obtain that 
\begin{align*}
\sum_{i=1}^Nh\Bigg(\tilde{\psi}^{(N)}_t+\frac{\gamma_u}{\gamma_l N}\sum_{j=1}^N|\tilde{U}^j_t-\breve{U}^j_t|+\breve{U}^i_t\Bigg)\geq \sum_{i=1}^Nh(\tilde{\psi}^{(N)}_t+\tilde{U}^{i}_t)=0,
\end{align*}
which implies that
\begin{align*}
\breve{\psi}^{(N)}_t\leq \tilde{\psi}^{(N)}_t+\frac{\gamma_u}{\gamma_l N}\sum_{j=1}^N|\tilde{U}^j_t-\breve{U}^j_t|.
\end{align*}
Then the first inequality of equation \eqref{eq4} follows from symmetry, which yields that
\begin{align}\label{psiN}
|\tilde{\psi}^{(N)}_t|\leq |x^l_t|+|\tilde{\psi}^{(N)}_t-x^l_t|\leq |x^l_t|+\frac{\gamma_u}{\gamma_l N}\sum_{j=1}^N|\tilde{U}^j_t|,
\end{align}
with $x^l_t$ being such that $h(x^l_t)=l_t$. 
By Doob's inequality and H\"{o}lder's inequality, there exists a constant $C$ independent of $N$, such that the first inequality in equation \eqref{psinphin} holds, and then similar estimate holds for $\tilde{\phi}^{(N)}$. Hence, $\tilde{\psi}^{(N)},\tilde{\phi}^{(N)}\in\mathcal{S}^2(\mathbb{F}^{(N)};\mathbb{R})$.

It remains to prove equation \eqref{separate} holds.  Noting that $h$ is bi-Lipschitz and increasing, we have
\begin{align*}
|l_t-u_t|&=\Bigg|\frac{1}{N}\sum_{i=1}^N h(\tilde{U}^{i}_t+\tilde{\psi}^{(N)}_t)-\frac{1}{N}\sum_{i=1}^N h(\tilde{U}^{i}_t+\tilde{\phi}^{(N)}_t)\Bigg|\\
&=\frac{1}{N}\sum_{i=1}^N \Big|h(\tilde{U}^{i}_t+\tilde{\psi}^{(N)}_t)-h(\tilde{U}^{i}_t+\tilde{\phi}^{(N)}_t)\Big|\\
&\leq \gamma_u\Big|\tilde{\psi}^{(N)}_t-\tilde{\phi}^{(N)}_t\Big|.
\end{align*}
The proof is complete.
\end{proof}

\begin{theorem}\label{thm3.1}
Suppose Conditions (B1) and (B2) hold. The IPS \eqref{eq7_2_s_tilde} has a unique solution $(\{\tilde{Y}^i\}_{1\leq i\leq N},\tilde{M}^{(N)},\tilde{K}^{(N)})\in \mathcal{S}^2(\mathbb{F}^{(N)};\mathbb{R}^N)\times \mathcal{M}_{loc}(\mathbb{F}^{(N)};\mathbb{R})\times \mathcal{A}_{BV}(\mathbb{F}^{(N)};\mathbb{R})$.
Furthermore, there exists a constant $C$ independent of $N$, such that, for all $1\leq i\leq N$,
\begin{align}
\label{eqn:thm3.1_estimate}
\E\Bigg[\sup_{t\in[0,T]}|\tilde{Y}^i_t|^2\Bigg]
\leq C\Bigg(1+\E|\xi|^2+\E\Bigg[\int_0^T|f(t)|^2dt\Bigg]\Bigg).
\end{align}
\end{theorem}

\begin{proof}
We first prove the uniqueness. Suppose that $(\{\tilde{\tilde{Y}}^i\}_{1\leq i\leq N},\tilde{\tilde{M}}^{(N)},\tilde{\tilde{K}}^{(N)})$ is another solution to the IPS \eqref{eq7_2_s_tilde}. Without loss of generality, we assume that there exists a pair $(i,t)\in\{1,\cdots,N\}\times[0,T)$, such that $\P(\tilde{A})>0$, where $\tilde{A}:=\{\tilde{\tilde{Y}}^i_t>\tilde{Y}^i_t\}$. Then, there exists a constant $\varepsilon>0$ such that  $\P(\tilde{B})>0$, where $\tilde{B}:=\{\tilde{\tilde{Y}}^i_t\geq \tilde{Y}^i_t+\varepsilon\}\subset \tilde{A}$.  Define the stopping time $\tilde{\tau}$ as follows:
\begin{align*}
\tilde{\tau}:=\inf\Big\{s\geq t:\tilde{\tilde{Y}}^i_s=\tilde{Y}^i_s\Big\}.
\end{align*}
Then, on the set $\tilde{A}$, we have $t<\tilde{\tau}\leq T$ and $\tilde{\tilde{Y}}^i_s>\tilde{Y}^i_s$ for $s\in[t,\tilde{\tau})$. Since $h$ is strictly increasing, on the set $\tilde{A}$, for any $s\in[t,\tilde{\tau})$ we have
\begin{align*}
u_t\geq \frac{1}{N}\sum_{j=1}^Nh(\tilde{\tilde{Y}}^j_s)>\frac{1}{N}\sum_{j=1}^Nh(\tilde{Y}^j_s)\geq l_t, 
\end{align*}
It follows from the Skorokhod condition that $d\tilde{\tilde{K}}^{(N),+}_s=0$ and $d\tilde{K}^{(N),-}_s=0$ for $s\in[t,\tilde{\tau})$ on the set $\tilde{A}$. Let $\{\sigma_n\}_{n\geq 1}$ be a nondecreasing sequence of $\mathbb{F}^{(N)}$-stopping times with $\P(\lim_{n\rightarrow\infty}\sigma_n =T)=1$ such that $\{\tilde{M}^{(N)}_{t\wedge \sigma_n}\}_{t\in[0,T]}$ and  $\{\tilde{\tilde{M}}^{(N)}_{t\wedge \sigma_n}\}_{t\in[0,T]}$ are martingales. Without loss of generality, we assume that $\sigma_1\geq t$. Set $\tilde{\tau}_n=\tilde{\tau}\wedge \sigma_n$. It is clear that 
\begin{align*}
\tilde{Y}^i_t-\tilde{\tilde{Y}}^i_t=&\tilde{Y}^i_{\tilde{\tau}_n}-\tilde{\tilde{Y}}^i_{\tilde{\tau}_n}-(\tilde{M}^{(N)}_{\tilde{\tau}_n}-\tilde{M}^{(N)}_t)+(\tilde{\tilde{M}}^{(N)}_{\tilde{\tau}_n}-\tilde{\tilde{M}}^{(N)}_t)+(\tilde{K}^{(N),+}_{\tilde{\tau}_n}-\tilde{K}^{(N),+}_t)\\
&-(\tilde{K}^{(N),-}_{\tilde{\tau}_n}-\tilde{K}^{(N),-}_t)-(\tilde{\tilde{K}}^{(N),+}_{\tilde{\tau}_n}-\tilde{\tilde{K}}^{(N),+}_t)+(\tilde{\tilde{K}}^{(N),-}_{\tilde{\tau}_n}-\tilde{\tilde{K}}^{(N),-}_t).
\end{align*}
Recalling that $d\tilde{\tilde{K}}^{(N),+}_s=0$ and $d\tilde{K}^{(N),-}_s=0$ for $s\in[t,\tilde{\tau})$ on the set $\tilde{A}$ and $\tilde{B}$ is a subset of $\tilde{A}$, we have
\begin{align*}
(\tilde{Y}^i_t-\tilde{\tilde{Y}}^i_t)\mathbbm{1}_{\tilde{B}}=&(\tilde{Y}^i_{\tilde{\tau}_n}-\tilde{\tilde{Y}}^i_{\tilde{\tau}_n})\mathbbm{1}_{\tilde{B}}-(\tilde{M}^{(N)}_{\tilde{\tau}_n}-\tilde{M}^{(N)}_t)\mathbbm{1}_{\tilde{B}}+(\tilde{\tilde{M}}^{(N)}_{\tilde{\tau}_n}-\tilde{\tilde{M}}^{(N)}_t)\mathbbm{1}_{\tilde{B}}\\
&+(\tilde{K}^{(N),+}_{\tilde{\tau}_n}-\tilde{K}^{(N),+}_t)\mathbbm{1}_{\tilde{B}}+(\tilde{\tilde{K}}^{(N),-}_{\tilde{\tau}_n}-\tilde{\tilde{K}}^{(N),-}_t)\mathbbm{1}_{\tilde{B}}.
\end{align*}
Taking expectations on both sides, noting that $\tilde{B}\in\mathcal{F}^{(N)}_t$ and $\tilde{\tau}_1\geq t$, we have
\begin{align*}
-\varepsilon \P(\tilde{B})\geq \E\Big[(\tilde{Y}^i_t-\tilde{\tilde{Y}}^i_t)\mathbbm{1}_{\tilde{B}}\Big]\geq\E\Big[(\tilde{Y}^i_{\tilde{\tau}_n}-\tilde{\tilde{Y}}^i_{\tilde{\tau}_n})\mathbbm{1}_{\tilde{B}}\Big]\rightarrow 0,\quad \textrm{ as } n\rightarrow \infty,
\end{align*}
which is a contradiction. Hence, $\tilde{Y}^i\equiv\tilde{\tilde{Y}}^i$ for $1\leq i\leq N$. Since $\tilde{Y}^i$ is an $\mathbb{F}^{(N)}$-semimartingale, its decomposition is unique. Therefore, we have $\tilde{K}^{(N)}=\tilde{\tilde{K}}^{(N)}$ and $\tilde{M}^{(N)}=\tilde{\tilde{M}}^{(N)}$.

Now, we aim to show the existence. By Theorem 3.7 in \cite{hamadene2005bsdes} and Lemma \ref{lempsiphi}, the doubly RBSDE:
\begin{equation}\label{DRBSDE1}
\begin{cases}
\tilde{S}_t=\Psi^{(N)}_T+\Phi^{(N)}_T-(\tilde{M}^{(N)}_T-\tilde{M}^{(N)}_t)+(\tilde{K}^{(N)}_T-\tilde{K}^{(N)}_t),\vspace{0.2cm}\\
\tilde{\psi}^{(N)}_t\leq \tilde{S}_t\leq \tilde{\phi}^{(N)}_t, \quad t\in[0,T],\vspace{0.2cm}\\
\tilde{K}^{(N)}=\tilde{K}^{(N),+}-\tilde{K}^{(N),-}\quad\textrm{with}\quad \tilde{K}^{(N),+},\tilde{K}^{(N),-}\in \mathcal{A}(\mathbb{F}^{(N)};\mathbb{R}),\vspace{0.2cm}\\
\int_0^T(\tilde{S}_t-\tilde{\psi}^{(N)}_t)d\tilde{K}^{(N),+}_t=\int_0^T(\tilde{\phi}^{(N)}_t-\tilde{S}_t)d\tilde{K}^{(N),-}_t=0,
\end{cases}
\end{equation}
admits a unique solution $(\tilde{S},\tilde{M}^{(N)},\tilde{K}^{(N)})\in \mathcal{S}^2(\mathbb{F}^{(N)};\mathbb{R})\times \mathcal{M}_{loc}(\mathbb{F}^{(N)};\mathbb{R})\times\mathcal{A}_{BV}(\mathbb{F}^{(N)};\mathbb{R})$. 
For $1\leq i\leq N$, set
\begin{align*}
\tilde{Y}^i_t=\tilde{U}^{i}_t+\tilde{S}_t, \qquad t\in[0,T].
\end{align*}
Then, we have
\begin{align*}
\tilde{Y}^i_t=\E\Bigg[\xi^i+\int_t^T f^i(s) ds\,\Bigg|\,\mathcal{F}^{(N)}_t\Bigg]+(\Psi^{(N)}_T+\Phi^{(N)}_T)-(M_T^{(N)}-\tilde{M}^{(N)}_t)+(\tilde{K}^{(N)}_T-\tilde{K}^{(N)}_t).
\end{align*}
We claim that $(\{\tilde{Y}^i\}_{1\leq i\leq N},\tilde{M}^{N},\tilde{K}^{(N)})$ is a solution to the IPS \eqref{eq7_2_s_tilde}.

For the constraints, since $h$ is nondecreasing and noting the fact that $\tilde{\psi}^{(N)}_t\leq \tilde{S}_t\leq \tilde{\phi}^{(N)}_t$ for $t\in[0,T]$, it is easy to check that
\begin{align*}
l_t=\frac{1}{N}\sum_{i=1}^N h(\tilde{U}^{i}_t+\tilde{\psi}^{(N)}_t)&\leq \frac{1}{N}\sum_{i=1}^N h(\tilde{U}^{i}_t+\tilde{S}_t)\\
&=\frac{1}{N}\sum_{i=1}^N h(\tilde{Y}^i_t)\leq \frac{1}{N}\sum_{i=1}^N h(\tilde{U}^{i}_t+\tilde{\phi}^{(N)}_t)=u_t.
\end{align*}
Applying the Skorokhod condition in the doubly RBSDE \eqref{DRBSDE1} yields that
\begin{align*}
\int_0^T\Bigg(\frac{1}{N}\sum_{i=1}^N h(\tilde{Y}^i_t)-l_t\Bigg)d\tilde{K}^{(N),+}_t
=&\int_0^T \Bigg(\frac{1}{N}\sum_{i=1}^N h(\tilde{U}^{i}_t+\tilde{S}_t)-l_t\Bigg)d\tilde{K}^{(N),+}_t\\
=&\int_0^T \Bigg(\frac{1}{N}\sum_{i=1}^N h(\tilde{U}^{i}_t+\tilde{S}_t)-l_t\Bigg)\mathbbm{1}_{\{\tilde{S}_t=\tilde{\psi}^{(N)}_t\}}d\tilde{K}^{(N),+}_t\\
=&\int_0^T\Bigg(\frac{1}{N}\sum_{i=1}^N h(\tilde{U}^{i}_t+\tilde{\psi}^{(N)}_t)-l_t\Bigg)\mathbbm{1}_{\{\tilde{S}_t=\tilde{\psi}^{(N)}_t\}}d\tilde{K}^{(N),+}_t\\
=&0.
\end{align*}
Similarly, we have
\begin{align*}
\int_0^T\frac{1}{N}\Bigg(u_t-\sum_{i=1}^N h(\tilde{Y}^i_t)\Bigg)d\tilde{K}^{(N),-}_t=0.
\end{align*}
At last, by the construction for $\tilde{Y}^i$ and Theorem 3.8 in \cite{hamadene2005bsdes}, we have
\begin{align*}
|\tilde{Y}^i_t|\leq |\tilde{U}^{i}_t|+|\tilde{S}_t|\leq \E\Bigg[\Bigg|\xi^i+\int_t^T f^i(s) ds\Bigg|\,\Bigg|\,\mathcal{F}^{(N)}_t\Bigg]+\E\Bigg[\sup_{t\in[0,T]}|\tilde{\psi}^{(N)}_t|+\sup_{t\in[0,T]}|\tilde{\phi}^{(N)}_t|\,\Bigg|\,\mathcal{F}^{(N)}_t\Bigg].
\end{align*}
The estimate for $\tilde{Y}^i$ in equation \eqref{eqn:thm3.1_estimate} follows from Doob's inequality, H\"{o}lder's inequality and the estimates \eqref{psinphin}. 
\end{proof}

\subsection{The particle system with non-constant driver}
\label{sec:nonlinear_nonconstant}
In this subsection, we consider the general IPS \eqref{eq7_2_tilde}.

\begin{theorem}\label{thm4.15}
Suppose Conditions (B1) and (B2) hold. The IPS \eqref{eq7_2_tilde} has a unique solution $(\{\widetilde{Y}^i\}_{1\leq i\leq N},\widetilde{M}^{(N)},\widetilde{K}^{(N)})\in \mathcal{S}^2(\mathbb{F}^{(N)},\mathbb{R}^N)\times\mathcal{M}_{loc}(\mathbb{F}^{(N)};\mathbb{R})\times \mathcal{A}_{BV}(\mathbb{F}^{(N)};\mathbb{R})$. 
\end{theorem}

\begin{proof}
For any $1\leq i\leq N$ and any given $\widetilde{V}^{i,(1)}\in \mathcal{S}^2(\mathbb{F}^{(N)},\mathbb{R})$, by Theorem \ref{thm3.1}, the following IPS:
\begin{displaymath}
\begin{cases}
\widetilde{Y}^{i,(1)}_t=\E\Big[\xi^i+\int_t^T f^i(s,\widetilde{V}^{i,(1)}_s) ds\,\Big|\,\mathcal{F}^{(N)}_t\Big]+(\Psi^{(N)}_T+\Phi^{(N)}_T)-(\widetilde{M}^{(N),(1)}_T-\widetilde{M}^{(N),(1)}_t)\\
\hspace{9cm}+\widetilde{K}^{(N),(1)}_T-\widetilde{K}^{(N),(1)}_t,\vspace{0.2cm}\\
l_t\leq \frac{1}{N}\sum_{i=1}^N h(\widetilde{Y}^{i,(1)}_t)\leq u_t, \quad t\in[0,T],\vspace{0.2cm}\\
\widetilde{K}^{(N),(1)}=\widetilde{K}^{(N),(1),+}-\widetilde{K}^{(N),(1),-}\quad\textrm{with}\quad \widetilde{K}^{(N),(1),+},\widetilde{K}^{(N,(1)),-}\in \mathcal{A}(\mathbb{F}^{(N)};\mathbb{R}),\vspace{0.2cm}\\
\int_0^T\Big(\frac{1}{N}\sum_{i=1}^N h(\widetilde{Y}^{i,(1)}_t)-l_t\Big)d\widetilde{K}^{(N,(1)),+}_t=\int_0^T (u_t-\frac{1}{N}\sum_{i=1}^N h(\widetilde{Y}^{i,(1)}_t))d\widetilde{K}^{(N),(1),-}_t=0,
\end{cases}
\end{displaymath}
 admits a unique solution $(\{\widetilde{Y}^{i,(1)}\}_{1\leq i\leq N},\widetilde{M}^{(N),(1)},\widetilde{K}^{(N),(1)})$.
We define the mapping $\widetilde{\Gamma}:\mathcal{S}^2(\mathbb{F}^{(N)},\mathbb{R})\rightarrow\mathcal{S}^2(\mathbb{F}^{(N)},\mathbb{R})$ by
\begin{align*}
\widetilde{\Gamma}(\widetilde{V}^{i,(1)})=\widetilde{Y}^{i,(1)}.
\end{align*}
Let $\widetilde{V}^{i,(2)}\in \mathcal{S}^2(\mathbb{F}^{(N)},\mathbb{R})$ and $\widetilde{\Gamma}(\widetilde{V}^{i,(2)})=\widetilde{Y}^{i,(2)}$. Then, we define
\begin{align*}
&\widetilde{U}^{i,(1)}_t:=\E\Bigg[\xi^i+\int_t^T f^i(s,\widetilde{V}_s^{i,(1)})ds\,\Bigg|\,\mathcal{F}^{(N)}_t\Bigg]\\
\text{and}\qquad &\widetilde{U}^{i,(2)}_t:=\E\Bigg[\xi^i+\int_t^T f^i(s,\widetilde{V}_s^{i,(2)})ds\,\Bigg|\,\mathcal{F}^{(N)}_t\Bigg].
\end{align*}
Let $\widetilde{\psi}^{(N),(1)}_t$, $\widetilde{\phi}^{(N),(1)}_t$, $\widetilde{\psi}^{(N),(2)}_t$, and $\widetilde{\phi}^{(N),(2)}_t$ be the $\mathcal{F}^{(N)}_t$-measurable random variable satisfying
\begin{align*}
&\frac{1}{N}\sum_{i=1}^N h(\widetilde{U}^{i,(1)}_t+\widetilde{\psi}^{(N),(1)}_t)=l_t, \qquad\quad \frac{1}{N}\sum_{i=1}^N h(\widetilde{U}^{i,(1)}_t+\widetilde{\phi}^{(N),(1)}_t)=u_t,\\
&\frac{1}{N}\sum_{i=1}^N h(\widetilde{U}^{i,(2)}_t+\widetilde{\psi}^{(N),(2)}_t)=l_t, \qquad\quad \frac{1}{N}\sum_{i=1}^N h(\widetilde{U}^{i,(2)}_t+\widetilde{\phi}^{(N),(2)}_t)=u_t.
\end{align*}
Let $\widetilde{S}^{(1)}$ (resp. $\widetilde{S}^{(2)}$) be the first component of the solution to the doubly RBSDE \eqref{DRBSDE1} with terminal value $(\Psi^{(N)}_T+\Phi^{(N)}_T)$, no driver, lower obstacle $\widetilde{\psi}^{(N),(1)}$ (resp. $\widetilde{\psi}^{(N),(2)}$) and upper obstacle $\widetilde{\phi}^{(N),(1)}_t$ (resp. $\widetilde{\phi}^{(N),(2)}_t$). By the proof of Theorem \ref{thm3.1}  and the representation for $\widetilde{S}^{(1)}$, $\widetilde{S}^{(2)}$ (see Theorem 3.8 in \cite{hamadene2005bsdes}), we obtain that
\begin{align*}
|\widehat{Y}^i_t|&\leq |\widehat{U}^i_t|+|\widehat{S}_t|\\
&\leq LT\E\Bigg[\sup_{s\in[0,T]}|\widehat{V}^i_s|\,\Bigg|\,\mathcal{F}^{(N)}_t\Bigg]+\E\Bigg[\sup_{s\in[0,T]}|\widehat{\psi}^{(N)}_s|\,\Bigg|\,\mathcal{F}^{(N)}_s\Bigg]+\E\Bigg[\sup_{s\in[0,T]}|\widehat{\phi}^{(N)}_s|\,\Bigg|\,\mathcal{F}^{(N)}_s\Bigg],
\end{align*}
where $\widehat{\Xi}_t=\Xi_t^{(1)}-\Xi_t^{(2)}$ for $\Xi=\widetilde{Y}^i,\widetilde{U}^{i},\widetilde{V}^i, \widetilde{S},\widetilde{\psi}^{(N)},\widetilde{\phi}^{(N)}$. By the Lipschitz continuity of $f$ in equation \eqref{eqn:Lipschitz}, we have
\begin{align*}
\frac{1}{N}\sum_{j=1}^N|\widehat{U}^j_t|\leq LT\E\Bigg[\frac{1}{N}\sum_{i=1}^N\sup_{t\in[0,T]}|\widehat{V}^i_t|\,\Bigg|\,\mathcal{F}^{(N)}_t\Bigg].
\end{align*}
By a similar analysis as \eqref{eq4}, we obtain that
\begin{align*}
	|\widehat{\psi}^{(N)}_t|\leq \frac{\gamma_u}{\gamma_l N}\sum_{j=1}^N|\widehat{U}^j_t|\quad\text{and}\quad	|\widehat{\phi}^{(N)}_t|\leq \frac{\gamma_u}{\gamma_l N}\sum_{j=1}^N|\widehat{U}^j_t|.
\end{align*}
 Then applying Doob's inequality and H\"{o}lder's inequality,
\begin{align*}
\E\Bigg[\sup_{t\in[0,T]}|\widehat{Y}^i_t|^2\Bigg]&\leq 12L^2T^2\E\Bigg[\sup_{t\in[0,T]}|\widehat{V}^i_t|^2\Bigg]+12\E\Bigg[\sup_{t\in[0,T]}|\widehat{\psi}^{(N)}_t|^2\Bigg]+12\E\Bigg[\sup_{t\in[0,T]}|\widehat{\phi}^{(N)}_t|^2\Bigg]\\
&\leq 12L^2T^2\E\Bigg[\sup_{t\in[0,T]}|\widehat{V}^i_t|^2\Bigg]+24\frac{\gamma_u^2}{\gamma_l^2}\E\Bigg[\sup_{t\in[0,T]}\Bigg(\frac{1}{N}\sum_{j=1}^N|\widehat{U}^j_t|\Bigg)^2\Bigg]\\
&\leq 12L^2T^2\E\Bigg[\sup_{t\in[0,T]}|\widehat{V}^i_t|^2\Bigg]+96L^2T^2\frac{\gamma_u^2}{\gamma_l^2}\E\Bigg[\frac{1}{N}\sum_{i=1}^N\sup_{t\in[0,T]}|\widehat{V}^i_t|^2\Bigg].
\end{align*}
Summing these inequalities yields that
\begin{align*}
\E\Bigg[\frac{1}{N}\sum_{i=1}^N\sup_{t\in[0,T]}|\widehat{Y}^i_t|^2\Bigg]\leq 12L^2 T^2\Bigg(1+8\frac{\gamma_u^2}{\gamma_l^2}\Bigg)\E\Bigg[\frac{1}{N}\sum_{i=1}^N\sup_{t\in[0,T]}|\widehat{V}^i_t|^2\Bigg].
\end{align*}
Choosing $T\leq \varepsilon$ with $\varepsilon>0$ being such that 
$12L^2 \varepsilon^2(1+8\gamma_u^2/\gamma_l^2)\leq \frac{1}{2},
$
then $\widetilde{\Gamma}$ is a contraction mapping. Thus, $\widetilde{\Gamma}$ has a unique fixed point in $\mathcal{S}^2(\mathbb{F}^{(N)},\mathbb{R})$ when $T\leq \varepsilon$, i.e., there exists a unique $\{\widetilde{Y}^i\}_{1\leq i\leq N}$ solving the IPS \eqref{eq7_2_tilde} for some $(\widetilde{M}^{(N)},\widetilde{K}^{(N)})\in\mathcal{M}_{loc}(\mathbb{F}^{(N)};\mathbb{R})\times\mathcal{A}_{BV}(\mathbb{F}^{(N)};\mathbb{R})$ on small time interval $[0,T]$. Note that $\widetilde{Y}^i$ is an $\mathbb{F}^{(N)}$-semimartingale, its decomposition is unique, which gives the uniqueness of $(\widetilde{M}^{(N)},\widetilde{K}^{(N)})$.

For the general case, let $n$ be a positive integer such that $T/n<\varepsilon$. For any $k=0,1,\cdots,n$, set $t_k=\frac{kT}{n}$. Let  $(\{\widetilde{Y}^{i,n}\}_{1\leq i\leq N},\widetilde{M}^{(N),n}, \widetilde{K}^{(N),n})$ be the unique solution to the IPS \eqref{eq7_2_tilde} on time interval $[t_{n-1},T]$. We define
\begin{align*}
\xi^{i,n-1}&=\E\Bigg[\xi^i+\int_{t_{n-1}}^T f^i(s,\widetilde{Y}^{i,n}_s) ds\,\Bigg|\,\mathcal{F}^{(N)}_{t_{n-1}}\Bigg],\\
\Psi^{(N),n-1}_{t_{n-1}}+\Phi^{(N),n-1}_{t_{n-1}}&=\widetilde{Y}^{i,n}_{t_{n-1}}-\xi^{i,n-1}\\
&=(\Psi^{(N)}_T+\Phi^{(N)}_T)-(\widetilde{M}^{(N),n}_T-\widetilde{M}^{(N),n}_{t_{n-1}})+\widetilde{K}^{(N),n}_T-\widetilde{K}^{(N),n}_{t_{n-1}}.
\end{align*}
It is easy to check that $\xi^{i,n-1}$, $\Psi^{(N),n-1}_{t_{n-1}}+\Phi^{(N),n-1}_{t_{n-1}}$ are $\mathcal{F}^{(N)}_{t_{n-1}}$-measurable and square integrable. Therefore, the following RBSDE on time interval $[t_{n-2},t_{n-1}]$: for any $1\leq i\leq N$, 
\begin{displaymath}
\begin{cases}
\widetilde{Y}^{i,n-1}_t=\E\Bigg[\xi^{i,n-1}+\int_t^{t_{n-1}} f^i(s,\widetilde{Y}^{i,n-1}_s) ds\,\Bigg|\,\mathcal{F}^{(N)}_t\Bigg]+(\Psi^{(N),n-1}_{t_{n-1}}+\Phi^{(N),n-1}_{t_{n-1}})\vspace{0.2cm}\\
\hspace{2cm} -(\widetilde{M}^{(N),n-1}_{t_{n-1}}-\widetilde{M}^{(N),n-1}_t)+\widetilde{K}^{(N),n-1}_{t_{n-1}}-\widetilde{K}^{(N),n-1}_t, \vspace{0.2cm}\\
l_t\leq \frac{1}{N}\sum_{i=1}^N h(\widetilde{Y}^{i,n-1}_t)\leq u_t, \vspace{0.2cm}\\
\widetilde{K}^{(N),n-1}=\widetilde{K}^{(N),n-1,+}-\widetilde{K}^{(N),n-1,-}\quad\textrm{with}\quad  \widetilde{K}^{(N),n-1,+},\widetilde{K}^{(N),n-1,-}\in \mathcal{A}(\mathbb{F}^{(N)};\mathbb{R}),\vspace{0.2cm}\\
\int_{t_{n-2}}^{t_{n-1}}(\frac{1}{N}\sum_{i=1}^N h(\widetilde{Y}^{i,n-1}_t)-l_t)d\widetilde{K}^{(N),n-1,+}_t=0,\vspace{0.2cm}\\
\int_{t_{n-2}}^{t_{n-1}} (u_t-\frac{1}{N}\sum_{i=1}^N h(\widetilde{Y}^{i,n-1}_t))d\widetilde{K}^{(N),n-1,-}_t=0,
\end{cases}
\end{displaymath}
admits a unique solution denoted by $(\{\widetilde{Y}^{i,n-1}\}_{1\leq i\leq N},\widetilde{M}^{(N),n-1}, \widetilde{K}^{(N),n-1})$.
We may define $(\{\widetilde{Y}^{i,k}\}_{1\leq i\leq N},\widetilde{M}^{(N),k}, \widetilde{K}^{(N),k})$ for $k=n-2,\cdots,1$, similarly. Set
\begin{align*}%
&\widetilde{Y}^i_t=\widetilde{Y}^{i,k}_t, \qquad \widetilde{M}^{(N)}_t=\widetilde{M}^{(N),k}_t+\sum_{l<k}\widetilde{M}^{(N),l}_{t_l},\\
& \widetilde{K}^{(N)}_t=\widetilde{K}^{(N),k}_t+\sum_{l<k}\widetilde{K}^{(N),l}_{t_l}, \qquad t\in[t_{k-1},t_k], \qquad k=1,2,\cdots,n.
\end{align*}
Then $(\{\widetilde{Y}^i\}_{1\leq i\leq N},\widetilde{M}^{(N)}, \widetilde{K}^{(N)})$ is the solution to the IPS \eqref{eq7_2_tilde}. The uniqueness follows from that on each small time interval.
\end{proof}

\begin{proposition}\label{prop4.2}
Suppose Conditions (B1) and (B2) hold.  For $(\{\widetilde{Y}^i\}_{1\leq i\leq N},\widetilde{M}^{(N)},\widetilde{K}^{(N)})$ being the solution of the IPS \eqref{eq7_2_tilde}, there exists a constant $C$ independent of $N$, such that for all $1\leq i\leq N$,
\begin{align}\label{estimateyi5}
\E\Bigg[\sup_{t\in[0,T]}|\widetilde{Y}^i_t|^2\Bigg]
\leq& C\Bigg(1+\E|\xi|^2+\E\left[\int_0^T|f(t,0)|^2dt\right]\Bigg).
\end{align}
\end{proposition}

\begin{proof} 
Set 
\begin{align}
	\label{eqn:tildetildeU}
  \widetilde{U}^{i}_t=\E\Bigg[\xi^i+\int_t^T f^i(s,\widetilde{Y}^{i}_s)ds\,\Bigg|\,\mathcal{F}^{(N)}_t\Bigg].
\end{align}
Let $  \widetilde{\psi}^{(N)}_t$ and $  \widetilde{\phi}^{(N)}_t$ be the $\mathcal{F}^{(N)}_t$-measurable random variables satisfying respectively
\begin{align*}
\frac{1}{N}\sum_{i=1}^N h(  \widetilde{U}^{i}_t+  \widetilde{\psi}^{(N)}_t)=l_t\quad\text{and}\quad \frac{1}{N}\sum_{i=1}^N h(  \widetilde{U}^{i}_t+  \widetilde{\phi}^{(N)}_t)=u_t.
\end{align*}
Consider $  \widetilde{S}$ as the first component of the solution to the doubly RBSDE as follows:
\begin{equation}\label{eqn:tildetildeS}
	\begin{cases}
		\widetilde{S}_t=\Psi^{(N)}_T+\Phi^{(N)}_T-(\widetilde{M}^{(N)}_T-\widetilde{M}^{(N)}_t)+(\widetilde{K}^{(N)}_T-\widetilde{K}^{(N)}_t),\vspace{0.2cm}\\
		\widetilde{\psi}^{(N)}_t\leq \widetilde{S}_t\leq \widetilde{\phi}^{(N)}_t, \quad t\in[0,T],\vspace{0.2cm}\\
		\widetilde{K}^{(N)}=\widetilde{K}^{(N),+}-\widetilde{K}^{(N),-}\quad\textrm{with}\quad \widetilde{K}^{(N),+},\widetilde{K}^{(N),-}\in \mathcal{A}(\mathbb{F}^{(N)};\mathbb{R}),\vspace{0.2cm}\\
		\int_0^T(\widetilde{S}_t-\widetilde{\psi}^{(N)}_t)d\widetilde{K}^{(N),+}_t=\int_0^T(\widetilde{\phi}^{(N)}_t-\widetilde{S}_t)d\widetilde{K}^{(N),-}_t=0.
	\end{cases}
\end{equation}
By the proof of Theorem \ref{thm3.1}  and the representation for $\widetilde{S}$ (see Theorem 3.8 in \cite{hamadene2005bsdes}), we have
\begin{align*}
\widetilde{Y}^i_t=  \widetilde{U}^{i}_t+  \widetilde{S}_t=  \widetilde{U}^{i}_t+\essinf_{\sigma\in\mathcal{T}^{(N)}_{t,T}}\esssup_{\tau\in\mathcal{T}^{(N)}_{t,T}}\E\Big[  \widetilde{R}^{N}(\sigma,\tau)\,\Big|\,\mathcal{F}^{(N)}_t\Big],
\end{align*}
where
\begin{align*}
  \widetilde{R}^{N}(\sigma,\tau):=(\Psi^{(N)}_T+\Phi^{(N)}_T)\mathbbm{1}_{\{\sigma\wedge \tau=T\}}+  \widetilde{\psi}^{(N)}_\tau \mathbbm{1}_{\{\tau<T,\tau\leq \sigma\}}+  \widetilde{\phi}^{(N)}_\sigma \mathbbm{1}_{\{\sigma<\tau\}}.
\end{align*}
Recalling that $  \widetilde{\psi}^{(N)}_T\leq \Psi^{(N)}_T+\Phi^{(N)}_T\leq   \widetilde{\phi}^{(N)}_T$, for any $t\in[0,T]$, 
\begin{align}\label{yit5}
|\widetilde{Y}^i_t|\leq &\E\Bigg[|\xi^i|+\int_t^T |f^i(s,0)|ds+L\int_t^T|\widetilde{Y}^i_s|ds\,\Bigg|\,\mathcal{F}^{(N)}_t\Bigg]\\
&+\E\Bigg[\sup_{s\in[t,T]}|  \widetilde{\psi}^{(N)}_s|+\sup_{s\in[t,T]}|  \widetilde{\phi}^{(N)}_s|\,\Bigg|\,\mathcal{F}^{(N)}_t\Bigg].\nonumber
\end{align}
By a similar analysis as the proof of equation \eqref{psiN}, we obtain that
\begin{align*}
|  \widetilde{\psi}^{(N)}_t|\leq |x^l_t|+\frac{\gamma_u}{\gamma_l N}\sum_{j=1}^N|  \widetilde{U}^{j}_t|\quad\text{and}\quad |  \widetilde{\phi}^{(N)}_t|\leq |x^u_t|+\frac{\gamma_u}{\gamma_l N}\sum_{j=1}^N|  \widetilde{U}^{j}_t|,
\end{align*}
where $x^l_t$ and $x^u_t$ satisfy respectively $h(x^l_t)=l_t$ and $h(x^u_t)=u_t$. Then there exists a constant $C$ independent of $N$, such that 
\begin{equation}\label{barpsin5}\begin{split}
&\E\Bigg[\sup_{s\in[t,T]}|  \widetilde{\psi}^{(N)}_s|^2\Bigg]\leq C\Bigg(1+\E\Bigg[|\xi|^2+\int_0^T |f(s,0)|^2ds\Bigg]+\E\Bigg[\int_t^T\frac{1}{N}\sum_{j=1}^N|\widetilde{Y}^j_s|^2ds\Bigg]\Bigg),\\
&\E\Bigg[\sup_{s\in[t,T]}|  \widetilde{\phi}^{(N)}_s|^2\Bigg]\leq C\Bigg(1+\E\Bigg[|\xi|^2+\int_0^T |f(s,0)|^2ds\Bigg]+\E\Bigg[\int_t^T\frac{1}{N}\sum_{j=1}^N|\widetilde{Y}^j_s|^2ds\Bigg]\Bigg).
\end{split}\end{equation}
Applying Doob's inequality and H\"{o}lder's inequality, by equations \eqref{yit5} and \eqref{barpsin5}, we have
\begin{align}\label{yit'5}
\E|\widetilde{Y}^i_t|^2\leq  C\Bigg(1+\E\Bigg[|\xi|^2+\int_0^T |f(s,0)|^2ds\Bigg]+\E\Bigg[\int_t^T\frac{1}{N}\sum_{j=1}^N|\widetilde{Y}^j_s|^2ds\Bigg]+\E\Bigg[\int_t^T |\widetilde{Y}^i_s|^2ds\Bigg]\Bigg).
\end{align}
Summing over $i$ yields that
\begin{align*}
\E\Bigg[\frac{1}{N}\sum_{i=1}^N|\widetilde{Y}^i_t|^2\Bigg]\leq C\Bigg(1+\E\Bigg[|\xi|^2+\int_0^T |f(s,0)|^2ds\Bigg]+\int_t^T\E\Bigg[\frac{1}{N}\sum_{j=1}^N|\widetilde{Y}^j_s|^2\Bigg]ds\Bigg).
\end{align*}
It follows from Gr\"onwall's inequality that 
\begin{align*}
\E\Bigg[\frac{1}{N}\sum_{i=1}^N|\widetilde{Y}^i_t|^2\Bigg]\leq C\Bigg(1+\E\Bigg[|\xi|^2+\int_0^T |f(s,0)|^2ds\Bigg]\Bigg).
\end{align*}
Plugging the above inequality into equations \eqref{barpsin5} and \eqref{yit'5}, using Gr\"onwall's inequality again, we obtain that 
\begin{equation}\label{barpsinyit5}\begin{split}
\E\Bigg[\sup_{s\in[0,T]}|  \widetilde{\psi}^{(N)}_s|^2\Bigg]&\leq C\Bigg(1+\E\Bigg[|\xi|^2+\int_0^T |f(s,0)|^2ds\Bigg]\Bigg),\\
\E\Bigg[\sup_{s\in[0,T]}|  \widetilde{\phi}^{(N)}_s|^2\Bigg]&\leq C\Bigg(1+\E\Bigg[|\xi|^2+\int_0^T |f(s,0)|^2ds\Bigg]\Bigg),\\
\E|\widetilde{Y}^i_t|^2&\leq C\Bigg(1+\E\Bigg[|\xi|^2+\int_0^T |f(s,0)|^2ds\Bigg]\Bigg).
\end{split}\end{equation}
Finally, combining equations \eqref{yit5} and \eqref{barpsinyit5}, applying Doob's inequality and H\"{o}lder's inequality, we obtain the desired result.
\end{proof}

\subsection{Approximation}
\label{sec:nonlinear_Approximation}
In this subsection, we use the IPS \eqref{eq7_2_tilde} to approximate the doubly MRBSDE \eqref{nonlinearyz} and establish the rate of convergence. 
Let $(Y^i,  M^i,K)$  be the solution to the following doubly MRBSDE:
\begin{equation}\label{nonlinearyz_approx}
	\begin{cases}
		Y^i_t=\xi^i+\int_t^T f^i(s,  Y^i_s)ds-(M^i_T-M^i_t)+K_T-K_t, \quad \E[h(Y^i_t)]\in [l_t,u_t],\quad  t\in[0,T], \vspace{0.2cm}\\
		K=K^+ -K^- \quad\textrm{with}\quad  K^+,K^-\in \mathcal{I}([0,T];\,\mathbb{R}),\vspace{0.2cm}\\
		\int_0^T (\E[h(  Y^i_t)]-l_t)dK_t^+=\int_0^T (u_t-\E[h(  Y^i_t)])dK^-_t=0.
	\end{cases}
\end{equation}
Then $(Y^i,  M^i,K)$ for $1\leq i\leq N$ are independent copies of $(Y,M,K)$, which is the unique solution to the doubly MRBSDE \eqref{nonlinearyz}. For $U^i_t$, $\Psi_T$, $\Phi_T$, $\psi_T$ and $\phi_T$ defined in Section \ref{sec:linear_Approximation}, we recall here as follows:
\begin{align*}	  
	&U^i_t=\E\Bigg[\xi^i+\int_t^Tf^i(s,  Y^i_s)ds \,\Bigg|\,\mathcal{F}^i_t\Bigg]=\E\Bigg[\xi^i+\int_t^Tf^i(s,  Y^i_s)ds \,\Bigg|\,\mathcal{F}^{(N)}_t\Bigg],\\
&\Psi_T=\inf\Big\{x\geq 0;\,\E[h(\xi^i+x)]\geq l_T\Big\},\quad\quad
	\Phi_T=\sup\Big\{x\leq 0;\,\E[h(\xi^i+x)]\leq u_T\Big\},
\end{align*}
and $\{  \psi_t\}_{t\in[0,T]}$ and $\{  \phi_t\}_{t\in[0,T]}$ satisfy respectively
\begin{align*}
	\E[h(  U^i_t+  \psi_t)]=l_t\quad\text{and}\quad \E[h(  U^i_t+  \phi_t)]=u_t.
\end{align*}

\begin{theorem}\label{thm4.35}
Let Conditions (B1) and (B2) hold. For $(Y^i,  M^i,K)$ being an
independent copy of the solution to the doubly MRBSDE \eqref{nonlinearyz} and $(\widetilde{Y}^i,\widetilde{M}^{(N)},\widetilde{K}^{(N)})$ solving the IPS \eqref{eq7_2_tilde}, the following results hold:
\begin{enumerate}[label=(\roman*)]
	\item 
	If $\sup_{t\in[0,T]}\E|  \widetilde{Z}^1_t|^4<\infty$ and $h$ is twice continuously differentiable with bounded derivatives, then there exists a constant $C$ independent of $N$, such that
	\begin{align*}
		\E\Bigg[\sup_{t\in[0,T]}|\widetilde{Y}^i_t -  Y^i_t|^2\Bigg]\leq CN^{-1}.
	\end{align*}
	
	\item 
	If $\xi$ is $p$-integrable, $f(\cdot,0)$ is progressively measurable with $\E[\int_0^T|f(t,0)|^pdt]<\infty$ and $\sup_{t\in[0,T]}\E|  \widetilde{Z}^1_t|^p<\infty$ for some $p>4$, then there exists a constant $C$ independent of $N$, such that
	\begin{align*}
		\E\Bigg[\sup_{t\in[0,T]}|\widetilde{Y}^i_t -  Y^i_t|^2\Bigg]\leq CN^{-1/2}.
	\end{align*}
\end{enumerate}
\end{theorem}

\begin{proof}
By the proof of Proposition \ref{prop4.2} and the representation for $Y^i$ (equation \eqref{baryi}), we have
\begin{align*}
|Y^i_t-\widetilde{Y}^i_t|\leq &|U^i_t-\widetilde{U}^{i}_t|+\E\Big[|\Phi_T-\Phi^{(N)}_T|\,\Big|\,\mathcal{F}^{(N)}_t\Big]+\E\Big[|\Psi_T-\Psi^{(N)}_T|\,\Big|\,\mathcal{F}^{(N)}_t\Big]\\
&+\E\Bigg[\sup_{s\in[t,T]}|\phi_s-  \widetilde{\phi}^{(N)}_s|\,\Bigg|\,\mathcal{F}^{(N)}_t\Bigg]+\E\Bigg[\sup_{s\in[t,T]}|\psi_s-  \widetilde{\psi}^{(N)}_s|\,\Bigg|\,\mathcal{F}^{(N)}_t\Bigg]\\
\leq &L\E\Bigg[\int_t^T|Y^i_s-  \widetilde{Y}^i_s|du\,\Bigg|\,\mathcal{F}^{(N)}_t\Bigg]+\E\Big[|\Phi_T-\Phi^{(N)}_T|\,\Big|\,\mathcal{F}^{(N)}_t\Big]+\E\Big[|\Psi_T-\Psi^{(N)}_T|\,\Big|\,\mathcal{F}^{(N)}_t\Big]\\
&+\E\Bigg[\sup_{s\in[t,T]}|\phi_s-  \widetilde{\phi}^{(N)}_s|\,\Bigg|\,\mathcal{F}^{(N)}_t\Bigg]+\E\Bigg[\sup_{s\in[t,T]}|\psi_s-  \widetilde{\psi}^{(N)}_s|\,\Bigg|\,\mathcal{F}^{(N)}_t\Bigg].
\end{align*}
By a similar analysis as equation \eqref{hatyi'}, there exists a constant $C$ independent of $N$, such that for any $r\in [0,T]$,
\begin{equation}\label{hatyi}
	\begin{split}
&\E\Bigg[\sup_{t\in[r,T]}|\widetilde{Y}^i_t -  Y^i_t|^2\Bigg]\\
&\leq C\Bigg\{\E\Bigg[\int_r^T|Y^i_s-  \widetilde{Y}^i_s|^2ds\Bigg]+\E\Bigg[\sup_{s\in[r,T]}|\phi^{(N)}_s-  \widetilde{\phi}^{(N)}_s|^2\Bigg]+\E\Bigg[\sup_{s\in[r,T]}|\psi^{(N)}_s-  \widetilde{\psi}^{(N)}_s|^2\Bigg]\\
&\hspace{4.3cm}+\E\Bigg[\sup_{s\in[0,T]}|\phi_s-  {\phi}^{(N)}_s|^2\Bigg]+\E\Bigg[\sup_{s\in[0,T]}|\psi_s-  {\psi}^{(N)}_s|^2\Bigg]\Bigg\},
\end{split}\end{equation}
where $\psi^{(N)}$, $\phi^{(N)}$ are given in equation \eqref{beforeconvergepsi}. Applying the estimates \eqref{eq4}, for any $t\in[r,T]$, we have 
\begin{align*}
&|\phi^{(N)}_t-  \widetilde{\phi}^{(N)}_t|\leq \frac{\gamma_u}{\gamma_l N}\sum_{j=1}^N |U^j_t-\widetilde{U}^{j}_t|\leq L\frac{\gamma_u}{\gamma_l }\E\Bigg[\frac{1}{N}\sum_{j=1}^N\int_r^T |\widetilde{Y}^j_s-  Y^j_s|du\,\Bigg|\,\mathcal{F}^{(N)}_t\Bigg],\\
&|\psi_t^{(N)}-  \widetilde{\psi}^{(N)}_t|\leq \frac{\gamma_u}{\gamma_l N}\sum_{j=1}^N |U^j_t-\widetilde{U}^{j}_t|\leq L\frac{\gamma_u}{\gamma_l }\E\Bigg[\frac{1}{N}\sum_{j=1}^N\int_r^T |\widetilde{Y}^j_s-  Y^j_s|du\,\Bigg|\,\mathcal{F}^{(N)}_t\Bigg].
\end{align*}
Plugging the above inequality into equation \eqref{hatyi}, using Doob's inequality and H\"{o}lder's inequality, we have that for any $r\in[0,T]$, 
\begin{align}\label{eq9}
\E\Bigg[\sup_{t\in[r,T]}|\widetilde{Y}^i_t -  Y^i_t|^2\Bigg]\leq C\Bigg\{\E\Bigg[\int_r^T|Y^i_s-  \widetilde{Y}^i_s|^2ds\Bigg]+\E\Bigg[\frac{1}{N}\sum_{j=1}^N\int_r^T |\widetilde{Y}^j_s-  Y^j_s|^2ds\Bigg]\Bigg\}+C I^{(N)},
\end{align}
where 
\begin{align*}
I^{(N)}:=&\E\Bigg[\sup_{s\in[0,T]}|\phi_s-  {\phi}^{(N)}_s|^2\Bigg]+\E\Bigg[\sup_{s\in[0,T]}|\psi_s-  {\psi}^{(N)}_s|^2\Bigg].
\end{align*}
Summing equation \eqref{eq9} over $i$, we obtain that 
\begin{align*}
\E\Bigg[\frac{1}{N}\sum_{i=1}^N\sup_{t\in[r,T]}|\widetilde{Y}^i_t -  Y^i_t|^2\Bigg]\leq &C\E\Bigg[\frac{1}{N}\sum_{j=1}^N\int_r^T |\widetilde{Y}^j_s-  Y^j_s|^2ds\Bigg]+C I^{(N)}\\
\leq &C\int_r^T\E\Bigg[ \frac{1}{N}\sum_{j=1}^N\sup_{s\in[u,T]}|\widetilde{Y}^j_s-  Y^j_s|^2\Bigg]du+C I^{(N)}.
\end{align*}
It follows from Gr\"onwall's inequality that
\begin{align*}
\E\Bigg[\frac{1}{N}\sum_{i=1}^N\sup_{t\in[0,T]}|\widetilde{Y}^i_t -  Y^i_t|^2\Bigg]\leq C I^{(N)}.
\end{align*}
Combining the above estimate with equation \eqref{eq9}, we finally obtain that
\begin{align*}
\E\Bigg[\sup_{t\in[0,T]}|\widetilde{Y}^i_t -  Y^i_t|^2\Bigg]\leq CI^{(N)},
\end{align*} 
where $C$ is a constant only depending on $L,m,M,T$.  By Lemma \ref{convergepsi}, we finish the proof of (i).

To prove (ii), it suffices to show that under Conditions (B1) and (B2), if $\xi$ is $p$-integrable, $f(\cdot,0)$ is progressively measurable with $\E[\int_0^T|f(t,0)|^pdt]<\infty$ and $\sup_{t\in[0,T]}\E|  \widetilde{Z}^1_t|^p<\infty$ for some $p>4$, then there exists a constant $C$ independent of $N$, such that 
		\begin{align}
			\label{eqn:lemma_part2}
			\E\Bigg[\sup_{s\in[0,T]}\Big|\psi_s-  {\psi}^{(N)}_s\Big|^2\Bigg]\leq CN^{-1/2}.
		\end{align}
It holds, since by equation (2.6) in \cite{Briand2020Particles} and equation \eqref{eq10} we have
\begin{align*}
	\E\Bigg[\sup_{s\in[0,T]}\Big|\psi_s-  {\psi}^{(N)}_s\Big|^2\Bigg]\leq \frac{\gamma_u^2}{\gamma_l^2}\E\Bigg[\sup_{t\in[0,T]}|W_1(\nu_t,\nu_t^{(N)})|^2\Bigg],
\end{align*}
where $W_1$ is the Wasserstein-$1$ distance defined by 
\begin{align*}
	W_1(\nu,\nu'):=\sup_{\varphi: 1\textrm{-Lipschitz}}\left|\int \varphi(d\nu-d\nu')\right|.
\end{align*}
Applying Theorem 10.2.7 in \cite{rachev2006mass} and the results about the control of Wasserstein distance of empirical measures of i.i.d. sample to the true law by \cite{Briand2020Particles}, we obtain the following estimate
\begin{align*}
	\E\Bigg[\sup_{t\in[0,T]}|W_1(\nu_t,\nu_t^{(N)})|^2\Bigg]\leq CN^{-1/2},
\end{align*}
which gives equation \eqref{eqn:lemma_part2} as desired. 
Analogously, we can show that
\begin{align*}
 \E\Bigg[\sup_{s\in[0,T]}\Big|\phi_s- {\phi}^{(N)}_s\Big|^2\Bigg]\leq CN^{-1/2},
\end{align*}
Now following the steps in the proof of (i), we could complete the proof of (ii).
\end{proof}


\bibliography{bib-ms}

\begin{thebibliography}{}

\bibitem[Bouchard et~al., 2015]{Bouchard2015BSDEs}
Bouchard, B., Elie, R., and R{\'e}veillac, A. (2015).
\newblock {BSDE}s with weak terminal condition.
\newblock {\em The Annals of Probability}, 43(2):572 -- 604.

\bibitem[Briand et~al., 2020a]{Briand2020Particles}
Briand, P., de~Raynal, P.-E.~C., Guillin, A., and Labart, C. (2020a).
\newblock Particles systems and numerical schemes for mean reflected stochastic
  differential equations.
\newblock {\em The Annals of Applied Probability}, 30(4):1884--1909.

\bibitem[Briand et~al., 2018]{briand2018bsdes}
Briand, P., Elie, R., and Hu, Y. (2018).
\newblock {BSDE}s with mean reflection.
\newblock {\em The Annals of Applied Probability}, 28(1):482--510.

\bibitem[Briand et~al., 2020b]{briand2020mean}
Briand, P., Ghannoum, A., and Labart, C. (2020b).
\newblock Mean reflected stochastic differential equations with jumps.
\newblock {\em Advances in Applied Probability}, 52(2):523--562.

\bibitem[Briand and Hibon, 2021]{briand2021particles}
Briand, P. and Hibon, H. (2021).
\newblock Particles systems for mean reflected {BSDE}s.
\newblock {\em Stochastic Processes and their Applications}, 131:253--275.

\bibitem[Cvitanic and Karatzas, 1996]{cvitanic1996backward}
Cvitanic, J. and Karatzas, I. (1996).
\newblock Backward stochastic differential equations with reflection and
  {D}ynkin games.
\newblock {\em The Annals of Probability}, pages 2024--2056.

\bibitem[El~Karoui et~al., 1997]{el1997reflected}
El~Karoui, N., Kapoudjian, C., Pardoux, E., Peng, S., and Quenez, M.-C. (1997).
\newblock Reflected solutions of backward {SDE}'s, and related obstacle
  problems for {PDE}'s.
\newblock {\em the Annals of Probability}, 25(2):702--737.

\bibitem[Falkowski and S{\l}omi{\'n}ski, 2021]{falkowski2021mean}
Falkowski, A. and S{\l}omi{\'n}ski, L. (2021).
\newblock Mean reflected stochastic differential equations with two
  constraints.
\newblock {\em Stochastic Processes and their Applications}, 141:172--196.

\bibitem[Falkowski and S{\l}omi{\'n}ski, 2022]{falkowski2022backward}
Falkowski, A. and S{\l}omi{\'n}ski, L. (2022).
\newblock Backward stochastic differential equations with mean reflection and
  two constraints.
\newblock {\em Bulletin des Sciences Math{\'e}matiques}, 176:103117.

\bibitem[Hamadene and Hassani, 2005]{hamadene2005bsdes}
Hamadene, S. and Hassani, M. (2005).
\newblock {BSDE}s with two reflecting barriers: the general result.
\newblock {\em Probability Theory and Related Fields}, 132:237--264.

\bibitem[Hamad{\`e}ne and Lepeltier, 2000]{hamadene2000reflected}
Hamad{\`e}ne, S. and Lepeltier, J.-P. (2000).
\newblock Reflected {BSDE}s and mixed game problem.
\newblock {\em Stochastic Processes and their Applications}, 85(2):177--188.

\bibitem[Hibon et~al., 2018]{hibon2017quadratic}
Hibon, H., Hu, Y., Lin, Y., Luo, P., and Wang, F. (2018).
\newblock Quadratic {BSDE}s with mean reflection.
\newblock {\em Mathematical Control and Related Fields}, 8(3-4):721--738.

\bibitem[Li, 2023]{li2023backward}
Li, H. (2023).
\newblock Backward stochastic differential equations with double mean
  reflections.
\newblock {\em arXiv preprint arXiv:2307.05947}.

\bibitem[Li, 2018]{li2018large}
Li, Y. (2018).
\newblock Large deviation principle for the mean reflected stochastic
  differential equation with jumps.
\newblock {\em Journal of Inequalities and Applications}, 2018:1--15.

\bibitem[Ma and Cvitani{\'c}, 2001]{ma2001reflected}
Ma, J. and Cvitani{\'c}, J. (2001).
\newblock Reflected forward-backward {SDE}s and obstacle problems with boundary
  conditions.
\newblock {\em Journal of Applied Mathematics and Stochastic Analysis},
  14(2):113--138.

\bibitem[Ning and Wu, 2021]{ning2021well}
Ning, N. and Wu, J. (2021).
\newblock Well-posedness and stability analysis of two classes of generalized
  stochastic volatility models.
\newblock {\em SIAM Journal on Financial Mathematics}, 12(1):79--109.

\bibitem[Ning and Wu, 2023]{ning2023multi}
Ning, N. and Wu, J. (2023).
\newblock Multi-dimensional path-dependent forward-backward stochastic
  variational inequalities.
\newblock {\em Set-Valued and Variational Analysis}, 31(1):2.

\bibitem[Qu and Wang, 2023]{qu2023multi}
Qu, B. and Wang, F. (2023).
\newblock Multi-dimensional {BSDE}s with mean reflection.
\newblock {\em Electronic Journal of Probability}, 28:1--26.

\bibitem[Rachev and R{\"u}schendorf, 2006]{rachev2006mass}
Rachev, S.~T. and R{\"u}schendorf, L. (2006).
\newblock {\em Mass transportation problems: Applications}.
\newblock Springer Science \& Business Media.

\bibitem[Ren and Wu, 2016]{ren2016approximate}
Ren, J. and Wu, J. (2016).
\newblock On approximate continuity and the support of reflected stochastic
  differential equations.
\newblock {\em The Annals of Probability}, 44(3):2064--2116.

\bibitem[Skorokhod, 1961]{skorokhod1961stochastic}
Skorokhod, A.~V. (1961).
\newblock Stochastic equations for diffusion processes in a bounded region.
\newblock {\em Theory of Probability \& Its Applications}, 6(3):264--274.

\bibitem[Sznitman, 1991]{sznitman1991topics}
Sznitman, A.-S. (1991).
\newblock Topics in propagation of chaos.
\newblock {\em Lecture notes in mathematics}, pages 165--251.

\end{thebibliography}



\end{document}